\documentclass[a4paper]{amsart}
\usepackage{color}
\usepackage{graphicx}
\usepackage{subfigure}
\usepackage{float}
\usepackage{url}
\usepackage{enumerate}
\usepackage{xspace}
\usepackage{multirow}
\usepackage{hhline}
\usepackage{lmodern}
\usepackage[justification=centering, font=small,
labelfont=bf]{caption}

\newcommand{\df}[1]{\ensuremath{\mathrm{df}({#1})}\xspace}
\newcommand{\core}{{\mathrm{core}}}

\newcommand{\Col}{{\mathrm{Col}}}

\newcommand{\odd}{\operatorname{odd}\xspace}
\newcommand{\scc}{\operatorname{scc}\xspace}

\theoremstyle{plain}
\newtheorem{theorem}{Theorem}[section]
\newtheorem{lemma}[theorem]{Lemma}

\newtheorem{corollary}[theorem]{Corollary}
\newtheorem{conjecture}{Conjecture}

\theoremstyle{definition}

\newtheorem{remark}[theorem]{Remark}

\newtheorem*{example*}{Example}

  \title[Cubic
  graphs of colouring defect 3 and conjectures of Berge and
  Alon-Tarsi]{Cubic
  graphs of colouring defect 3 and\\conjectures of Berge and
  Alon-Tarsi}  %

\author[J. Karab\'a\v{s}]{J\'an Karab\'a\v{s}}
\email[J. Karab\'a\v{s}]{karabas@savbb.sk}
\author[E. M\'a\v cajov\'a]{Edita M\'a\v cajov\'a}
\email[E. M\'a\v cajov\'a]{macajova@dcs.fmph.uniba.sk}
\author[R. Nedela]{Roman Nedela}
\email[R. Nedela]{nedela@savbb.sk}
\author[M. \v Skoviera]{Martin \v Skoviera}
\email[M. \v Skoviera]{skoviera@dcs.fmph.uniba.sk}

\address[J. Karab\'a\v{s}]{
Institute of Computer Science and Mathematics, FEEIT,
Slovak University of Technology, Bratislava, Slovakia
}

\address[E. M\'a\v cajov\'a, M. \v Skoviera]{
Comenius University, Bratislava, Slovakia
}

\address[R. Nedela]{
Faculty of Applied Sciences, University of West Bohemia,
Pilsen, Czech Republic}

\address[J. Karab\'a\v{s}, R. Nedela]{
Mathematical Institute of Slovak Academy of Sciences,
Bansk\'a Bystrica, Slovakia
}

\begin{document}
\maketitle

\begin{abstract}
We study two measures of uncolourability of cubic graphs, their
colouring defect and perfect matching index. The colouring
defect of a cubic graph $G$ is the smallest number of edges
left uncovered by three perfect matchings; the perfect matching
index of $G$ is the smallest number of perfect matchings that
together cover all edges of~$G$. We provide a complete
characterisation of cubic graphs with colouring defect $3$
whose perfect matching index is greater or equal to $5$. The
result states that every such graph arises from the Petersen
graph with a fixed $6$-cycle $C$ by substituting edges or
vertices outside $C$ with suitable $3$-edge-colourable cubic
graphs. Our research is motivated by two deep and long-standing
conjectures, Berge's conjecture stating that five perfect
matchings are enough to cover the edges of any bridgeless cubic
graph and the shortest cycle cover conjecture of Alon and Tarsi
suggesting that every bridgeless graph can have its edges
covered with cycles of total length at most $7/5\cdot m$, where
$m$ is the number of edges. We apply our characterisation to
showing that every cubic graph with colouring defect~$3$ admits
a cycle cover of length at most $4/3\cdot m +1$, where $m$ is
the number of edges, the bound being achieved by the graphs
whose perfect matching index equals $5$. We further prove that
every snark containing a $5$-cycle with an edge whose
endvertices removed yield a $3$-edge-colourable graph has a
cycle cover of length at most $4/3\cdot m+1$, as well.


\end{abstract}

\section{Introduction}
\noindent{}Many deep problems in graph theory, including
celebrated long-standing conjectures about cycles, perfect
matchings, and flows, can be reduced to cubic graphs. The
validity of these conjectures often depends upon a fairly
narrow class of $2$-connected cubic graphs that do not admit a
proper $3$-edge-colouring, the \emph{snarks}. In this paper, we
focus on properties of snarks that can be expressed in terms of
perfect matchings and cycles. In particular, we deal with two
important conjectures, the Berge conjecture about the size of a
perfect-matching cover and the $7/5$-conjecture about the
length of a shortest cycle cover.


\textbf{Berge's conjecture.} According to a stronger form of
Petersen's perfect matching theorem every edge of a bridgeless cubic graph
belongs to a perfect matching. It follows that the edges of any
bridgeless cubic graph $G$ can be covered with a collection of
perfect matchings. The minimum number of perfect matchings
needed to cover the edges of $G$ is called the \emph{perfect
matching index} of $G$, denoted by $\pi(G)$. Berge's conjecture
(circa 1970) suggests that $\pi(G)\leq 5$ for every bridgeless
cubic graph $G$. However, neither a constant upper bound for
$\pi(G)$ nor an example with $\pi(G)>5$ are currently known.


\textbf{The 7/5-conjecture.} In 1985, Alon and Tarsi \cite{AT},
and independently Jaeger \cite{Raspaud-diss}, conjectured that
every bridgeless graph can have its edges covered with a
collection of cycles of total length at most $7/5\cdot m$ where
$m$ is the number of edges. This conjecture is surprisingly
strong because, if true, it would imply the validity of the
cycle double cover conjecture (Jamshy and Tarsi \cite{JT}). The
conjecture is best possible because it is attained for the
Petersen graph. Fleischner's splitting lemma
\cite[Lemma~III.26]{F} implies that the $7/5$-conjecture
reduces to bridgeless subcubic graphs. However, as with other
conjectures concerning cycles and flows, the case of cubic
graphs appears to be crucial. Observe that every cycle cover of
a cubic graph must have length at least $4/3\cdot m$, because
every vertex is incident with at least one doubly covered edge.
This lower bound is reached by $3$-edge-colourable cubic
graphs, where any two bi-coloured $2$-factors provide a cycle
cover of length $4/3\cdot m$. For snarks, the conjecture
remains widely open. It is known, however, that every snark
with perfect matching index $4$ admits a cycle cover of length
$4/3\cdot m$ (see Steffen \cite[Theorem~3.1]{S2}). Potential
counterexamples to the $7/5$-conjecture may thus occur only
among snarks with perfect matching index at least $5$.

\textbf{Colouring defect of cubic graphs.} Both conjectures
discussed above are notoriously hard open problems. Therefore
it is meaningful to verify their validity for restricted but
nontrivial infinite families of snarks. One such family arises
by considering the colouring defect of snarks. For a bridgeless
cubic graph $G$, its \emph{colouring defect}, denoted by
$\df{G}$, is the smallest number of edges that are left
uncovered by any set of three perfect matchings. Clearly, the
colouring defect of every $3$-edge-colourable graph equals
zero. On the other hand, every snark has colouring defect at
least~$3$. Moreover, if the colouring defect equals~$3$, then
the three uncovered edges alternate with the doubly covered
ones on a chordless $6$-cycle~\cite{S2}.

In this paper we verify the Berge conjecture and the
$7/5$-conjecture for cubic graphs of colouring defect $3$, a
family that is, in a certain sense, closest to
$3$-edge-colourable graphs. Computational results reveal that
the overwhelming majority of snarks of small order (up to 36
vertices) have colouring defect 3 (see \cite[p.~28,
Table~1]{KMNS-defred}). It is conceivable that this ratio will
approach 100$\%$ with increasing order, which would indicate
that almost all snarks have colouring defect 3. Our
investigation thus covers a significant portion of snarks.
\begin{figure}[h!]
 \centering
 \includegraphics[scale=1.2]{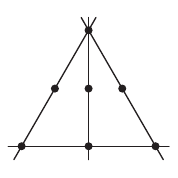}
\caption{Four lines of the Fano plane forming the configuration
$F_4$.}
\label{fig:F_4}
\end{figure}

Snarks of colouring defect $3$ can be alternatively described
by means of edge colourings where colours are the points of the
Fano plane and colour patterns around vertices form lines of
the configuration $F_4$ of four lines of the Fano plane
covering all seven points (see Figure~\ref{fig:F_4}). Every
colouring that minimises the use of the unique point of $F_4$
that belongs to three lines of $F_4$ forces the existence of a
$6$-cycle and thus restricts cyclic connectivity to $6$; see
\cite[Section~3]{KMNS-defred}. This phenomenon, where the
existence of a certain type of an edge colouring restricts
cyclic connectivity to $6$, clearly points towards the
Jaeger-Swart conjecture \cite{JSw} that there are no cyclically
$7$-edge-connected snarks.


\textbf{Main results.} Results of this paper pertain to several
areas discussed above. Proofs extensively use tools and
techniques developed in our recent papers \cite{KMNS-Berge}
and~\cite{KMNS-defred}, nevertheless, our present paper is
self-contained.

In Theorem~\ref{thm:decomp-def3} we prove that every snark $G$
with colouring defect $3$ admits a decomposition
$\mathcal{D}=\{G_1,\ldots, G_m\}$ such that all decomposition
factors are cyclically $4$-edge-connected except possibly
$G_1$, which has colouring defect $3$ while all other members
of $\mathcal{D}$ are $3$-edge-colourable. If $G_1$ is not
cyclically $4$-edge-connected, then it contains a triangle $T$
whose contraction produces a cyclically $4$-edge-connected
graph $G/T$, but the colouring defect of $G/T$ can be
arbitrarily large. Triangles whose contraction increases
defect, called \emph{essential}, indeed exist as shown in
\cite[Theorem~5.3]{KMNS-defred}.

Our main interest naturally focuses on those cubic graphs with
colouring defect $3$ whose perfect matching index is $5$ or
more. In \cite{KMNS-Berge} we have proved that every cubic
graph with colouring defect $3$ has has perfect matching index
at most $5$.  We have further shown that if the graph is
cyclically $4$-edge-connected, then its perfect matching index
equals $4$ unless it is the Petersen graph, whose perfect
matching index equals $5$.

In this paper, we aim to characterise all cubic graphs with
colouring defect~$3$ whose perfect matching index equals $5$,
irrespectively of their cyclic connectivity. In
Theorem~\ref{thm:main} we prove that a connected cubic graph
$G$ with colouring defect $3$ has perfect matching index $5$ if
and only if it arises from the Petersen graph with a specified
$6$-cycle $C$ by inserting edge-deleted connected
$3$-edge-colourable graphs into certain edges outside $C$ and
by substituting certain vertices outside $C$ with
vertex-deleted $3$-edge-colourable quasi-bipartite cubic graphs
(see Section~\ref{sec:pmi5} for the definition). An important
step towards the proof of Theorem~\ref{thm:main} is a rather
surprising fact that every cubic graph with colouring defect
$3$ containing an essential triangle has perfect matching index
$4$ (see Theorem~\ref{thm:triancore} and
Corollary~\ref{cor:esstrian}).

We apply Theorem~\ref{thm:main} and the methods developed for
its proof to the 7/5-conjecture. We prove that every snark with
colouring defect $3$ has a cycle cover of length at most
$4/3\cdot m+1$ where $m$ is the number of edges
(Theorem~\ref{thm:scc-def3}). The bound is best possible, being
achieved by the graphs whose perfect matching index equals $5$.
Our result improves that of Steffen \cite[Corollary~2.20]{S2}
which states that every triangle-free cubic graph with
colouring defect $3$ has a cycle cover of length at most
$4/3\cdot m+2$.

Some of our methods apply to cubic graphs with colouring defect
greater than $3$. Our final result,
Theorem~\ref{thm:pentagons}, states that every snark containing
a $5$-cycle with an edge whose endvertices removed yield a
$3$-edge-colourable graph admits a cycle cover of length at
most $4/3\cdot m+1$, where $m$ is the number of edges.


\textbf{Structure of the paper.} The next section lists basic
definitions used throughout the paper. The concept of the
colouring defect of a cubic graph is explored in Section~3.
Sections~4 and~5 prepare necessary tools for Section~6, which
provides a characterisation of graphs with colouring defect $3$
and perfect matching index at least $5$. Results of the
preceding sections are applied to the $7/5$-conjecture in
Section~7. The paper is closed with several remarks about the
presented results.

\section{Preliminaries}\label{sec:prelim}

\noindent{}All graphs in this paper are finite and for the most
part cubic (3-valent). Multiple edges and loops are permitted.
By a \emph{cycle} in a graph $G$ (more precisely, a
$\mathbb{Z}_2$-cycle) we mean a subgraph with all vertices of
even degree. The \textit{length} of a cycle is the number of
its edges. We use the term \emph{circuit} to mean a connected
$2$-regular graph. An $k$-\emph{cycle} is a circuit of
length~$k$. The length of a shortest circuit in a graph is its
\emph{girth}. If $H$ is an induced subgraph of $G$, we let
$G/H$ denote the graph arising from $G$ by contracting each
component of $H$ into a single vertex.

Subgraphs are often identified with their edge sets. For a
subgraph $Y$, or just for a set of vertices of a graph $G$, we
let $\delta_G(Y)$ denote the edge cut consisting of all edges
joining $Y$ to vertices not in $Y$. A connected graph $G$ is
said to be \emph{cyclically $k$-edge-connected} if the removal
of fewer than $k$ edges from $G$ cannot create a graph with at
least two components containing circuits. An edge cut $S$ in
$G$ that separates two circuits from each other is
\emph{cycle-separating}.

An \emph{edge colouring} of a graph $G$ is a mapping from the
edge set of $G$ to a set of colours. A colouring is
\emph{proper} if any two edge-ends incident with the same
vertex carry distinct colours. A \emph{$k$-edge-colouring} is a
proper colouring where the set of colours has $k$ elements. 
A $2$-connected cubic with no $3$-edge-colouring is called a \emph{snark}.

Our definition leaves the concept of a snark as wide as
possible since more restrictive definitions may lead to
overlooking certain important phenomena in cubic graphs. Our
definition thus follows Cameron et al. \cite{CCW}, Nedela and
\v Skoviera \cite{NS-decred}, Steffen \cite{S1}, and others,
rather than a more common approach where snarks are required to
be cyclically $4$-edge-connected and have girth at least $5$,
see for example~\cite{FMS-survey}. In this paper, such snarks
are called \emph{nontrivial}.

It is often convenient to regard the colours $1$,
$2$, and $3$ as elements of the group
$\mathbb{Z}_2\times\mathbb{Z}_2$, identifying them with their binary representation: $1=(0,1)$,
$2=(1,0)$, and $3=(1,1)$. The colours of the edges
around any vertex $v$ are pairwise distinct exactly when they sum to zero. The latter condition coincides with the Kirchhoff
law for nowhere-zero $\mathbb{Z}_2\times\mathbb{Z}_2$-flows, which implies that a proper $3$-edge-colouring of a cubic graph coincides the a nowhere-zero
$\mathbb{Z}_2\times\mathbb{Z}_2$-flow on it.

The following lemma is a variant of the classsical Parity Lemma.

\begin{lemma}{\rm (Parity Lemma)}\label{lem:par}
Let $G$ be a cubic graph and let $H$ be a subgraph of $G$. If 
$\xi$ is a proper $3$-edge-colouring of $H\cup\delta_G(H)$, then
$$\sum_{e\in\delta_G(H)}\xi(e)=0.$$
Equivalently, the number of edges in $\delta_G(H)$ carrying any
fixed colour has the same parity as the size of the cut.
\end{lemma}

A set of vertices or an induced subgraph $H$ of a snark $G$ is
called \emph{non-removable} if $G - V(H)$ is colourable;
otherwise, $H$ is \emph{removable}. It well known (and it
readily follows from Parity Lemma) that digons, triangles, and
quadrilaterals in snarks are removable subgraphs. It is also
known that a pair $\{u,v\}$ of adjacent vertices of a snark $G$
forming an edge $e=uv$ is removable if and only if $e$ is
\emph{suppressible}, which means that the cubic graph $G\sim e$
homeomorphic to $G-e$ is not $3$-edge-colourable (see
\cite[Proposition~4.2]{NS-decred}).

\section{Colouring defect of a snark}\label{sec:arrays}

\noindent{}In this section we develop a theory that unfolds
from the concept of the colouring defect of a cubic graph in
the extent necessary for the subsequent sections.

A set $\mathcal{M}=\{M_1,M_2, M_3\}$ consisting of three
perfect matchings of a bridgeless cubic graph $G$ will be
called a \emph{$3$-array} of perfect matchings of $G$. We think
of $\mathcal{M}$ as a generalisation of a $3$-edge-colouring,
in which case the three perfect matchings are pairwise
disjoint. We say that an edge is \emph{uncovered}, \emph{simply
covered}, \emph{doubly covered}, or \emph{triply covered} if it
belongs, respectively, to zero, one, two, or three members
of~$\mathcal{M}$. The minimum number of edges left uncovered by
a $3$-array of perfect matchings is the \emph{colouring defect}
of $G$, denoted by $\df{G}$. For brevity, we often drop the
word ``colouring" and speak simply of the \emph{defect} of~$G$.
A $3$-array that leaves $\df{G}$ uncovered edges will be called
\emph{optimal}.

An important structure related to a $3$-array is its
\emph{core}, the subgraph induced by all the edges that
are not simply covered. The core of a $3$-array $\mathcal{M}$
will be denoted by $\core(\mathcal{M})$. A core will be called
\emph{optimal} whenever $\mathcal{M}$ is optimal.
Figure~\ref{fig:petersen_core} shows an optimal $3$-array
$\{M_1,M_2,M_3\}$ of perfect matchings of the Petersen graph.
The label of each edge $e$ lists the indices of the perfect
matchings containing  $e$. The core of this $3$-array is the
``outer'' hexagon of the Petersen graph.

\begin{figure}[h!]
 \centering
 \includegraphics[scale=1.3]{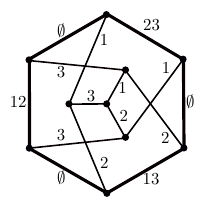}
 \caption{An optimal $3$-array of the Petersen graph;
 its core is displayed bold.}
\label{fig:petersen_core}
\end{figure}

Let $\mathcal{M}=\{M_1, M_2, M_3\}$ be a $3$-array for a cubic
graph $G$. To each edge $e$ of $G$ we  assign the list
$\phi(e)$ of elements of the set $\{1,2,3\}$  in lexicographic
order so that an element $j\in\{1,2,3\}$ is listed if and only
if $e\in M_j$. This assignment defines an edge colouring
$$\phi\colon E(G)\to\{\emptyset, 1, 2, 3, 12, 13, 23, 123\}.$$
If we take into account the fact that each index
$i\in\{1,2,3\}$ represents a perfect matching, we can easily
deduce that $\phi$ is a proper colouring if and only if $G$
contains no triply covered edge, that is, if the colour $123$
is never used.

Assume that $G$ has no triply covered edge with respect to the
$3$-array $\mathcal{M}$. Since each index $i\in\{1,2,3\}$
represents a perfect matching, the colouring $\phi$ permits
only four combinations of colours around any vertex, namely
$\{1,2,3\}$, $\{\emptyset, 1,23\}$, $\{\emptyset, 2,13\}$, and
$\{\emptyset, 3,12\}$. These four triples form a point-line
configuration isomorphic to the configuration $F_4$ of four
lines of the Fano plane covering all seven points. Due to this
fact, the colouring $\phi$ will henceforth be referred to as
the \emph{Fano colouring} associated with $\mathcal{M}$.
Figure~\ref{fig:Fanocol} on the left displays the configuration
$F_4$ with lines corresponding to admissible combinations of
colours around vertices. The right-hand side of the figure
shows the same configuration with the standard projective
coordinates in
$\mathbb{Z}_2\times\mathbb{Z}_2\times\mathbb{Z}_2-\{0\}$.

\begin{figure}[h!]
\subfigure[]{
\includegraphics[scale=1.3]{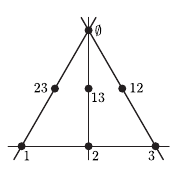}\label{fig:konf1}
}\qquad
\subfigure[]{
\includegraphics[scale=1.3]{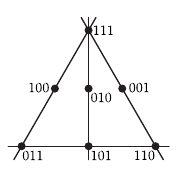}\label{fig:konf2} }
\caption{The configuration $F_4$ for $3$-arrays with no triply
covered edge.}
\label{fig:Fanocol}
\end{figure}

Recall that the coordinates of any three points constituting a
line sum to zero. It follows that we can regard a Fano
colouring (including the case of a triply covered edge) as a
flow with values in the group
$\mathbb{Z}_2\times\mathbb{Z}_2\times\mathbb{Z}_2$ and call it
a \emph{Fano flow} associated with~$\mathcal{M}$. This flow is
nowhere-zero if and only if $G$ has no triply covered edge.

The rest of this section is devoted to cubic graphs of defect
$3$, which, as we shall see, is the smallest nonzero value that
an cubic graph can have. For the missing proofs we
refer the reader to \cite[Section~4]{KMNS-defred}.

\begin{figure}[h]
\includegraphics[scale=1.5]{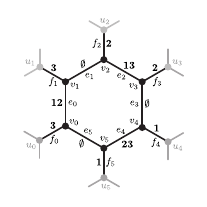}
\label{fig:oddness1}\caption{A hexagonal core in a snark and its immediate
neighbourhood.}
\label{fig:core3}
\end{figure}

The following theorem provides a detailed description of cores
in snarks of defect $3$.

\begin{theorem}\label{thm:hexcore}\cite[Theorem~3.3]{KMNS-defred}
Every snark $G$ has colouring defect at least $3$. Furthermore,
the following three statements are equivalent for any cubic
graph $G$.
\begin{enumerate}[{\rm(i)}]
\item $\df{G}=3$.
\item The core of any optimal $3$-array of $G$ is a
    $6$-cycle.
\item $G$ contains an induced $6$-cycle $C$ such that the
    subgraph $G-E(C)$ admits a proper $3$-edge-colouring
    under which the six edges of $\delta_G(C)$ receive
    colours $1,1,2,2,3,3$ or $1,2,2,3,3,1$ with respect to
    a fixed cyclic ordering induced by an orientation of
    $C$.
\end{enumerate}
\end{theorem}

If $G$ is a snark with colouring defect $3$, then by
Theorem~\ref{thm:hexcore}(iii) it contains an optimal array
$\mathcal{M}=\{M_1,M_2,M_3\}$ whose core is an induced
$6$-cycle. Such a core will be called a \emph{hexagonal core}
of $G$. Figure~\ref{fig:core3} illustrates the general
situation around any hexagonal core in a snark of defect $3$.

Next, we discuss possible mutual positions of hexagonal
cores and short circuits in snarks of defect $3$. The fact that
a hexagonal core is an induced $6$-cycle implies that it cannot
intersect a digon. The forthcoming two lemmas deal with
triangles and quadrilaterals, respectively. The first of them
features the concept of an essential triangle, which is crucial
for understanding cubic graphs of defect $3$. A triangle $T$ in
a snark $G$ with $\df{G}=3$ is said to be \emph{essential} if
the graph $G/T$ obtained by contracting $T$ to a vertex has
$\df{G/T}\ge 4$.

\begin{lemma}\label{lem:essential}\cite[Lemma~4.1]{KMNS-defred}
Let $G$ be a snark with $\df{G}=3$, let $C$ be a hexagonal core
of $G$, and let $T$ be an arbitrary triangle in $G$. The
following statements hold.
\begin{itemize}
\item[{\rm (i)}] If $C\cap T=\emptyset$, then $T$ is not
    essential.

\item[{\rm (ii)}]  If $C\cap T\ne\emptyset$, then $C\cap T$
    consists of a single uncovered edge, and the edge cut
    $\delta_G(T)$ comprises three pairwise independent edges.

\item[{\rm (iii)}] Every hexagonal core intersects at most
    one triangle.

\item[{\rm (iv)}] $G$ has at most one essential triangle.
 \end{itemize}
\end{lemma}

\begin{lemma}\label{lem:core_quadrangle}\cite[Lemma
4.5]{KMNS-defred}
Let $G$ be a snark with $\df{G}=3$. If a hexagonal core of $G$
intersects a quadrilateral, then the intersection consists of a
single uncovered edge. Moreover, two edges leaving the
quadrilateral are doubly covered and the other two are simply
covered.
\end{lemma}

In the following lemma we describe intersections of hexagonal
cores with small cycle-separating edge cuts. Statement~(i)
constitutes part of Lemma~4.2 in \cite{KMNS-defred}.
Statement~(ii) is extracted from the proof Lemma~4.3 of
\cite{KMNS-defred}; its proof is identical with that of Claim~1
in the mentioned lemma.

\begin{lemma}\label{lem:2+3-cuts}
The following two statements hold in every snark $G$ with
$\df{G}=3$.
\begin{itemize}
\item[{\rm (i)}] No hexagonal core of $G$ intersects a
    $2$-edge-cut.
\item[{\rm (ii)}] If a hexagonal core $C$ intersects a
    cycle-separating $3$-edge-cut $R$, then $C\cap R$
    consists of two doubly covered edges. Moreover, the
    graph $G-R$ has unique component $Q$ such that $G/Q$ is
    a snark. The intersection $C\cap Q$ consists of a
    single uncovered edge.
\end{itemize}
\end{lemma}

Essential triangles in snarks of colouring defect~3 are closely
related to several other concepts which we are about to
introduce. A pentagon $D$ in a cubic graph $G$ is called
\emph{heavy} if $D$ contains an edge $e=uv$ such that
$G-\{u,v\}$ is $3$-edge-colourable, or equivalently, if the
edge $e$ is non-suppressible. Our
terminology is derived from the behaviour of the function
$\psi_G\colon E(G)\to\mathbb{N}$ introduced by K\'aszonyi in
\cite{Ka2-nonplan, Ka3-struct} and by Bradley in
\cite[Definition~3.4]{Brad3} and defined by setting
$18\cdot\psi_G(x)$ to be the number of $3$-edge-colourings of
$G\sim x$. The K\'aszonyi function is constant on each pentagon
and is positive if and only if the pentagon is heavy
\cite{Brad1}. In particular, each edge of a heavy pentagon 
is non-suppressible.

The constant value of $\psi_G$ extends from a
$5$-cycle to the entire \emph{cluster of $5$-cycles}, which is
an inclusion-wise maximal connected subgraph of $G$ formed by
the union of $5$-cycles. For more information about K\'aszonyi
function we refer the reader to Bradley
\cite{Brad1}-\cite{Brad3} or K\'aszonyi
\cite{Ka1-ortho}-\cite{Ka3-struct}.

\medskip

The final result of this section implies that an essential
triangle can only arise only by inflating a vertex of a heavy
pentagon in a snark with defect greater than $3$. The result
strengthens Theorem 5.2 from \cite{KMNS-defred}.

\begin{theorem}\label{thm:inflation}
Let $G$ be a snark and let $G^v$ be the graph obtained from $G$
by inflating a vertex $v$ to a triangle $T$. Then $G^v$ has a
hexagonal core intersecting $T$ if and only if $v$ lies on a
heavy pentagon of $G$.
\end{theorem}

\begin{figure}[h!]
 \centering
\includegraphics[scale=0.9]{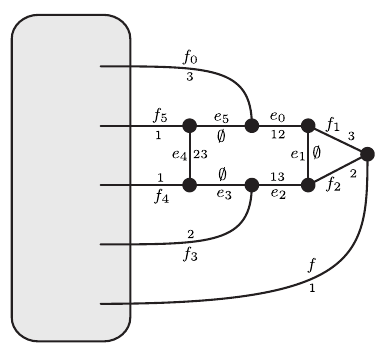}
 \caption{A triangle intersected by a hexagonal core.}
\label{fig:essential3}
\end{figure}

\begin{proof}
$(\Rightarrow)$ Assume that $C$ is a hexagonal core in $G^v$
that intersects $T$.  We show that $v$ lies on heavy $5$-cycle
of $G$. Lemma~\ref{lem:essential}(ii) implies that $C\cap T$
consists of a single uncovered edge, which we may assume to be
the edge $e_1$ indicated in Figure~\ref{fig:core3}. It follows
that the Fano colouring $\phi$ around $C$ is as illustrated in
Figure~\ref{fig:essential3}. Let us contract $T$ back to the
vertex $v$ and keep the colours of the edges of $G$. Clearly,
$C/T=(e_0e_2e_3e_4e_5)$ is a $5$-cycle containing $v$. Now,
consider the graph $G-\{v_0,v\}$. Set $\phi'(e_3)=3$,
$\phi'(e_4)=2$ and $\phi'(e)=\phi(e)$ for all remaining edges
of $G-\{v_0,v\}$. It is easy to see that $\phi'$ is a
$3$-edge-colouring, which readily implies that $C/T$ is a heavy
pentagon.

$(\Leftarrow)$ Assume that $v$ lies on a heavy $5$-cycle $D$ of
$G$. We first show that $D$ is an induced $5$-cycle. Let $G_1$
and $G_2$ be the subgraphs of $G$ induced by $V(D)$ and by
$V(G)-V(D)$, respectively. If $D$ is not induced, then
$|\delta_G(G_1)|=1$ or $|\delta_G(G_1)|=3$. The former
possibility is excluded, because $G$ is $2$-connected. Suppose
that $|\delta_G(G_1)|=3$. In this case, $G_1$ is a $5$-cycle
with one chord. Take any edge $e$ of $D$. Since $D$ is heavy,
$G\sim e$ is $3$-edge-colourable. Consequently, $G_2$ is
$3$-edge-colourable, too, because $G_2\subseteq G\sim e$.
However, one can easily extend any $3$-edge-colouring of $G_2$
to a $3$-edge-colouring of $G$, which is a contradiction. Thus
$D=(d_0d_1d_2d_3d_4)$ is an induced $5$-cycle. We may clearly
assume that $v$ is incident with the edges $d_4$ and~$d_0$. Let
$g_i$ be the edge of $\delta_G(D)$ adjacent to $d_{i-1}$ and
$d_i$, taking the indices modulo~$5$. The edge of $\delta_G(D)$
incident with $v$ is therefore~$g_0$. From Parity Lemma we
deduce that every $3$-edge-colouring $\sigma$ of $G-E(D)$
colours three of the edges in $\delta_G(D)$ with the same
colour, say~$1$, and the remaining two with colours $2$ and
$3$, respectively. Moreover, if $\sigma(g_i)=2$ and
$\sigma(g_j)=3$, then $g_i$ and $g_j$ are not adjacent to the
same edge of~$D$, otherwise $\sigma$ would extend to a
$3$-edge-colouring of $G$. In fact, it can be deduced from
Lemmas~6.1, 6.2, and~6.3 (iii) of \cite{NS-decred} that for any
two edges $g_i$ and $g_j$ of $\delta_G(D)$ that are not
adjacent to the same edge of $D$ there exists a
$3$-edge-colouring $\tau$ of $G-E(D)$ such that $\tau(g_i)=2$
and $\tau(g_j)=3$. Set $\tau(g_1)=2$ and $\tau(g_4)=3$, so that
all the remaining edges of $\delta_G(D)$ receive colour $1$,
see Figure~\ref{fig:5to6} (left).

\begin{figure}[h!]
 \centering
\includegraphics[width=0.6\textwidth]{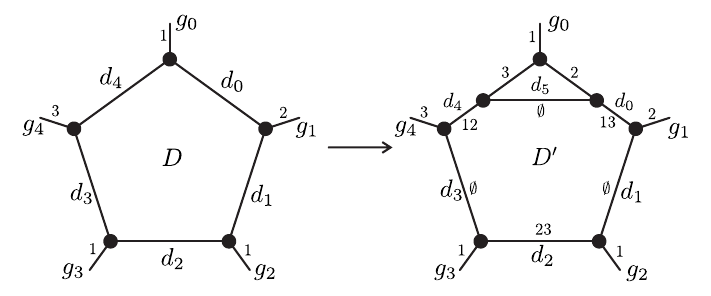}
\caption{Creating a hexagonal core by inflating a vertex on a
heavy pentagon.} \label{fig:5to6}
\end{figure}

Let us inflate $v$ into a triangle $T$, thereby producing the
graph $G^{v}$. Note that $G^{v}$ is a snark. For each edge of
$G$ incident with $v$ there is a unique edge of $G^{v}$
corresponding to it that leaves the triangle~$T$; we let the
latter edge inherit its name from the former. Having made this
agreement, let $d_5$ denote the edge of $T$ adjacent to both
$d_0$ and $d_4$. Clearly, $D'=(d_0d_1d_2d_3d_4d_5)$ is an
induced $6$-cycle of $G$. Now we can extend the colouring
$\tau$ of $G-E(D)$ to a Fano colouring of $G^v$. It is
sufficient to set $\tau(d_5)=\emptyset$ and derive the colours
of all other edges of $D'\cup T$ appropriately, see
Figure~\ref{fig:5to6} (right). It is easy to check that $\tau$
induces a $3$-array $\mathcal{N}$ of $G$ with
$\core(\mathcal{N})=D'$. Therefore $G$ has defect $3$ and
contains a hexagonal core that intersects a triangle. The proof
is complete.
\end{proof}

The previous theorem and Lemma~\ref{lem:essential} imply the
following.

\begin{corollary}\label{thm:essential}
Let $G$ be a snark with $\df{G}\ge 4$, let $v$ be a vertex of
$G$, and let $G^v$ be the snark created from $G$ by inflating
$v$ to a triangle. Then $\df{G^v}=3$ if and only if $v$ lies in
a heavy pentagon.
\end{corollary}

\section{Decomposition theorems}\label{sec:decomposition}

\noindent{}In this section we prove two decomposition theorems,
Theorems~\ref{thm:decomposition} and~\ref{thm:decomp-def3}. The
first of them states that every $2$-connected cubic graph $G$
can be uniquely decomposed into cyclically $4$-edge-connected
graphs in such a way that every vertex of $G$ falls into
precisely one decomposition factor. The second decomposition
theorem concerns specifically snarks of defect~3. The former
probably belongs to mathematical folklore at least in its
existence part. The proof of the latter significantly depends
on the uniqueness of the decomposition established in the
previous result, which is why we provide a detailed proof.

Consider an arbitrary $2$-edge-connected cubic graph $G$. If
$G$ is not cyclically $4$-edge-connected, then $G$ contains a
cycle-separating $3$-edge-cut or $2$-edge-cut. In fact, $G$
contains an \emph{independent} $2$-edge-cut or $3$-edge-cut --
one that consists of pairwise independent edges -- because
every minimum size cycle-separating edge cut in a cubic graph
is independent. For convenience, let us call an independent
edge cut $R$ \emph{small} whenever $2\le |R|\le 3$; note that
an independent edge cut in a cubic graph is automatically
cycle-separating.

Consider a $2$-connected cubic graph $G$ with a small
independent edge cut $R$. Let $L$ be a component of $G-R$. If
$|R|=2$, we define $\bar L$ to be the cubic graph obtained by
adding to $L$ an edge joining its two $2$-valent vertices and
call $\bar L$ the \emph{completion} of $L$ to a cubic graph. If
$|R|=3$, the \emph{completion} of $L$ to a cubic graph is
obtained by adding to $L$ a new vertex $v_L$ and joining $v_L$
to the three $2$-valent vertices of $L$. Thus, for every small
independent edge cut $R$ in $G$ whose removal leaves components
$L_1$ and $L_2$ we can create two smaller cubic graphs $\bar
L_1$ and $\bar L_2$ uniquely determined by $R$ and call the set
$\{\bar L_1, \bar L_2\}$ the \emph{decomposition} of $G$
along~$R$. It is easy to see that both $\bar L_1$ and $\bar
L_2$ are $2$-connected, and both are $3$-connected whenever
$G$~is.

We may continue by further decomposing $\bar L_1$ and $\bar
L_2$, if possible. If any of them contains a small independent
edge cut, we can decompose the graph along such a cut again and
continue as long as possible. The resulting set
$\mathcal{D}=\{G_1,\ldots,G_m\}$ consists of $2$-connected
cubic graphs that contain no cycle-separating $2$-edge-cuts and
$3$-edge-cuts. In other words, each member $G_i\in\mathcal{D}$
is a cyclically $4$-edge-connected cubic graph. We say that
$\mathcal{D}$ is a \emph{decomposition} of $G$ into cyclically
$4$-edge-connected cubic graphs.
In order to keep track of intermediate stages of the
decomposition process we define a \emph{decomposition} of $G$
along small independent edge cuts recursively as follows:
\begin{itemize}
\item[(1)] The singleton $\{G\}$ is a decomposition of $G$.
\item[(2)] If $\{H_1,\ldots,H_q\}$ is a
    decomposition of $G$ and one of its members $H_j$ has a
    decomposition $\{K,L\}$ along a small independent
edge cut, then the set $\{H_1,\ldots, H_{j-1},\penalty0 K, L,    H_{j+1}, \ldots, H_q\}$ is a  decomposition of
    $G$.
\end{itemize}

Observe that every vertex of $G$ belongs to exactly one member
$H_i$ of any decomposition~$\mathcal{B}=\{H_1,\ldots,H_q\}$ of
$G$. The elements of $\mathcal{B}$ will be called
\emph{decomposition factors}. Furthermore, we can reassemble
$G$ from the graphs $H_1, \ldots, H_q$ by means of two standard
operations, a $2$-sum and a $3$-sum.

Let $H$ and $K$ be $2$-connected cubic graphs with
distinguished edges $e$ and $f$, respectively. We define a
\emph{$2$-sum}\emph{ $H\oplus_2 K$} to be a cubic graph
obtained by deleting $e$ and $f$ and connecting the $2$-valent
vertices of $H-e$ to those of $K-f$. The two edges joining
$H-e$ to $K-f$ form an edge cut in $H\oplus_2 K$, called the
\emph{principal} $2$-edge-cut of $H\oplus_2 K$. If instead of
distinguished edges we choose distinguished vertices $u$ and
$v$ of $H$ and $K$, respectively, we can similarly define a
\emph{$3$-sum $G\oplus_3 H$} and the \emph{principal}
$3$-edge-cut of $H\oplus_3 K$. The resulting graphs
$H\oplus_2 K$ and $H\oplus_3 K$ are not uniquely determined,
nevertheless it is easy to see that they are always
$2$-connected.

If $\{H_1,H_2\}$ is a decomposition of a $2$-connected cubic
graph $G$ along a $2$-edge-cut $R$, then $2$-sum can be
performed in such a way that $G=H_1\oplus_2 H_2$ and the
principal $2$-edge-cut coincides with $R$. Similarly, if
$\{H_1,H_2\}$ is a decomposition of $G$ along an independent
$3$-edge-cut $S$, then one can perform $3$-sum in a way that
$G=H_1\oplus_3 H_2$ and the principal $3$-edge-cut coincides
with $S$.

Note that if $\{H_1,H_2\}$ is a decomposition of $G$ along an
independent $3$-edge-cut $S$, then each of $H_1$ and $H_2$ must
have at least four vertices, otherwise one of the components of
$G-S$ would be an isolated vertex, in which case the cut would
not be cycle-separating. Conversely, if $S$ is not
cycle-separating, then one of $H_1$ and $H_2$, say $L_2$, is
isomorphic to the \emph{$3$-dipole} $D_3$, by which we mean the
graph consisting of two vertices joined by three parallel
edges. Consequently, $G=H_1\oplus_3 H_2$ is isomorphic to
$H_1$, which shows that $D_3$ behaves like a neutral element
for the $3$-sum operation. A $3$-sum $H_1\oplus_3 H_2$ where
both $H_1$ and $H_2$ are different from the dipole $D_3$ will
be called \emph{nontrivial}. The definition of a nontrivial
$3$-sum still permits the situation that one or both
distinguished vertices are incident with a pair of parallel
edges. If this is not the case, the $3$-sum will be called
\emph{proper}. It is easy to see that the principal
$3$-edge-cut of a proper $3$-sum is independent.

Summing up, every $2$-connected cubic graph $G$ admits a
decomposition into cyclically $4$-edge-connected cubic graphs
by repeatedly applying decomposition along small independent
edge cuts. The decomposition process can be reversed and the
original graph $G$ can be obtained from the decomposition
factors by iterated $2$-sums and proper $3$-sums.

Every decomposition of a $2$-connected cubic graph $G$
along a small independent edge cut determines the two resulting
graphs uniquely up to isomorphism. We now prove that the entire
decomposition is also uniquely determined by $G$,
irrespectively of the order in which edge cuts are taken during
the decomposition process.

\begin{theorem}\label{thm:decomposition}
Every $2$-connected cubic graph $G$ admits a decomposition into
a collection $\mathcal{D}=\{G_1,\ldots,G_m\}$ of cyclically
$4$-edge-connected cubic graphs such that $G$ can be
reconstructed from them by a repeated application of $2$-sums
and proper $3$-sums. This collection is unique
up to ordering and isomorphism.
\end{theorem}

\begin{proof}
The existence of a decomposition into cyclically
$4$-edge-connected graphs has been established above. It
remains to prove its uniqueness. The statement is obviously
true if $G$ is cyclically $4$-edge-connected, that is, if it
has no small independent edge cuts. We therefore consider a
$2$-connected graph $G$ with at least one small independent
edge cut. Let $\mathcal{B}$ be an arbitrary decomposition of
$G$ into two or more smaller graphs. We say that a small
independent edge cut $R$ of $G$ is the \emph{leading cut} of
$\mathcal{B}$ if it is the first cut in the sequence of edge
cuts applied during the decomposition process that results in
$\mathcal{B}$. Clearly, it is  enough to show that any
decomposition of $G$ into cyclically $4$-edge-connected graphs
will contain the same collection of graphs no matter which
leading cut we have chosen.

We first prove the following.

\medskip\noindent
Claim~1. \emph{Let $R$ and $R'$ be any two distinct small
independent edge cuts of $G$. Then $G$ admits decompositions
$\mathcal{B}=\{G_1,G_2,G_3\}$ and
$\mathcal{B}'=\{G_1',G_2',G_3'\}$ with leading cuts $R$
and~$R'$, respectively,  such that $G_i$ is isomorphic to
$G_i'$ for each $i\in\{1,2,3\}$. }

\begin{figure}[h!]
	\centering
	\includegraphics[scale=1.1]{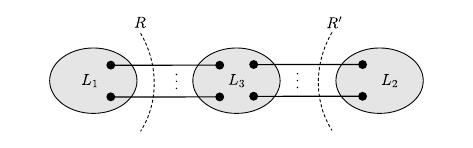}
\caption{Disjoint small independent edge cuts in the proof of
Claim~1}
	\label{fig:decomp_0}
\end{figure}

\noindent \emph{Proof of Claim~1.} First assume that $R\cap
R'=\emptyset$. The graph $G-(R\cup R')$ then consists of three
components $L_1$, $L_2$, and $L_3$, where $\delta_G(L_1)=R$,
$\delta_G(L_2)=R'$, and $\delta_G(L_3)=R\cup R'$, see
Figure~\ref{fig:decomp_0}. If we
decompose $G$ along $R$ first, we obtain a decomposition
$\{\bar L_1, H\}$ with $R'$ being inherited into~$H$. We
proceed by decomposing $H$ along $R'$ to get a decomposition
$\mathcal{A}=\{\bar L_1, \bar L_2, K\}$ of $G$ with $K\supseteq
L_3$. If, on the other hand, we first decompose $G$ along $R'$,
we obtain a decomposition $\{\bar L_2, H'\}$ where $H'$
inherits $R$.  We continue by decomposing $H'$ along $R$ to
obtain a decomposition $\mathcal{A}'=\{\bar L_2, \bar L_1,
K'\}$ of $G$ where $K'\supseteq L_3$. It is easy to see that
$K'$ and $K$ are isomorphic. Thus we can reorder $\mathcal{A}'$
to get the required isomorphisms between the members of
$\mathcal{A}$ and those of $\mathcal{A}'$.

Now assume that $R$ and $R'$ intersect.
Up to symmetry, one of the following four cases occurs:
\begin{itemize}
\item[(a)] $|R|=|R'|=2$ and $|R\cap R'|=1$,
\item[(b)] $|R|=2$, $|R'|=3$, and $|R\cap R'|=1$,
\item[(c)] $|R|=|R'|=3$ and $|R\cap R'|=1$, and
\item[(d)] $|R|=|R'|=3$ and $|R\cap R'|=2$.
\end{itemize}

In each of the four cases $G$ decomposes into three induced
subgraphs $L_1$, $L_2$, and $L_3$ joined between each other by
the edges of $R_1\cup R_2$. Without loss of generality we may
assume that $R=\delta_G(L_1)$ and $R'=\delta_G(L_2)$ and
that $|\delta_G(L_1)|\leq |\delta_G(L_2)|\leq |\delta_G(L_3)|$.
Since the cuts $R$ and $R'$ are independent, each of $L_1$,
$L_2$, and $L_3$ contains at least two vertices. The structure
of $G$
in Cases (a)-(c) is illustrated in
Figures~\ref{fig:decomp_a}-\ref{fig:decomp_c}. In Case (d), the
components of $G-(R\cup R')$ are the same as for
Case~(b), except that this time we consider intersecting edge
cuts $R=\delta_G(L_2)$ and $R'=\delta_G(L_3)$.

\begin{figure}[h!]
\centering
\includegraphics[scale=1.1]{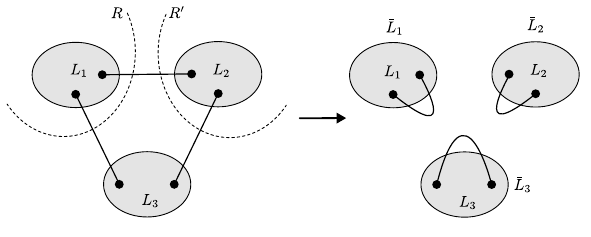}
\caption{Case (a)}
\label{fig:decomp_a}
\end{figure}

In Case (a), we have
$|\delta_G(L_1)|=|\delta_G(L_2)|=|\delta_G(L_3)|=2$. It follows
that
$$
G=
\bar L_1\oplus_2 (\bar L_2\oplus_2 \bar L_3)
=
\bar L_2\oplus_2 (\bar L_1\oplus_2 \bar L_3),
$$
see Figure~\ref{fig:decomp_a}. If we first decompose $G$ along
$R=\delta_G(L_1)$, we get the decomposition $\{\bar L_1,\bar
L_2\oplus_2 \bar L_3\}$. We continue by decomposing $\bar
L_2\oplus_2 \bar L_3$ along its principal edge cut to produce
the decomposition $\mathcal{A}=\{\bar L_1, \bar L_2, \bar
L_3\}$. If we take $R'=\delta_G(L_2)$ to be the leading cut and
proceed similarly, we obtain the decomposition
$\mathcal{A}'=\{\bar L_2, \bar L_1, \bar L_3\}$. Clearly,
$\mathcal{A}$ and $\mathcal{A}'$ with a changed order satisfy
the statement of Claim~1.

\begin{figure}[h!]
\centering
\includegraphics[scale=1.1]{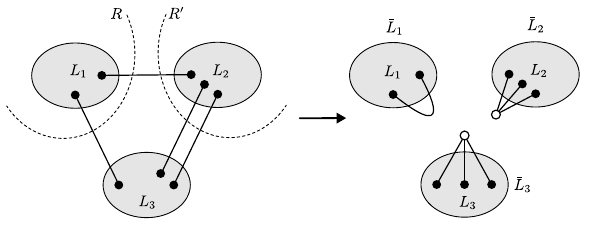}
\caption{Case (b)}
\label{fig:decomp_b}
\end{figure}

In Case (b) we have $|\delta_G(L_1)|=2$ and $|\delta_G(L_2)|=
|\delta_G(L_3)|=3$. It follows that
$$
G=
\bar L_1\oplus_2 (\bar L_2\oplus_3 \bar L_3)
=
\bar L_2\oplus_3 (\bar L_1\oplus_2 \bar L_3),
$$
see Figure~\ref{fig:decomp_b}. The main difference from
Case~(a) is that $\bar L_2$ and $\bar L_3$ have been created
from $L_2$ and $L_3$, respectively, by adding a vertex.
Otherwise, we can proceed analogously to prove that
irrespectively of the choice of the leading cut we end up with
the same decomposition $\{\bar L_1, \bar L_2, \bar L_3\}$ up to
the order of the graphs and isomorphism.

In Case (c) we have $|\delta_G(L_1)|=|\delta_G(L_2)|=3$ and
$|\delta_G(L_3)|=4$. Since both $R$ and $R'$ are independent
edge cuts, $L_1$ and $L_2$ have three $2$-valent vertices each,
while $L_3$ has four. We complete $L_1$ and $L_2$ to cubic
graphs $\bar L_1$ and $\bar L_2$ by adding a new vertex to each
as in Case (b). As regards $L_3$, we realise that the edges of
$\delta_G(L_3)$ naturally split to two pairs
$\delta_G(L_3)\cap\delta_G(L_1)$ and
$\delta_G(L_3)\cap\delta_G(L_2)$. Their endvertices in $L_3$,
which are of degree~$2$ in~$L_3$, split to the corresponding
two vertex pairs $\{u_1,u_2\}$ and $\{v_1,v_2\}$. We now extend
$L_3$ to a cubic graph $L_3^\sharp$ by adding two new adjacent
vertices $u$ and $v$ and join $u$ to both $u_1$ and $u_2$, and
$v$ to both $v_1$ and $v_2$, see Figure~\ref{fig:decomp_c}.
It is easy to see that regardless of whether we choose the leading
cut to be $R$ or $R'$ we will produce the same decomposition
$\{\bar L_1, \bar L_2, L_3^\sharp\}$ up to the order of the
graphs and isomorphism.

\begin{figure}[h!]
\centering
\includegraphics[scale=1.1]{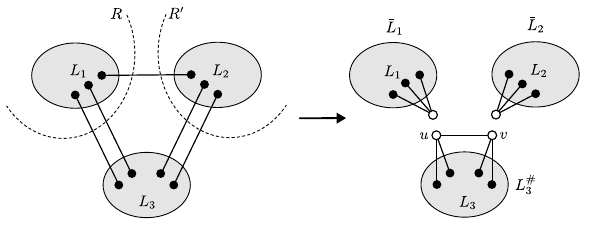}
\caption{Case (c)}
\label{fig:decomp_c}
\end{figure}

As previously mentioned, the structure of components of
$G-(R\cup R')$ in Case (d) is the same as for Case~(b), except
that this time we are considering edge cuts $R=\delta_G(L_2)$
and $R'=\delta_G(L_3)$ with $|R|=|R'|=3$ and $|R\cap R'|=2$. If
we start decomposing $G$ along the cut $R=\delta_G(L_2)$, we
obtain a decomposition $\{\bar L_2, \bar L_1\oplus_2 \bar
L_3\}$, where the distinguished edge of the $2$-sum is an edge
incident with the added vertex of $\bar L_3$. Although the
edges of $\delta_G(L_3)$ still constitute a $3$-edge-cut in
$\bar L_1\oplus_2 \bar L_3$, they are no more independent and
therefore this cut cannot be used for a further decomposition.
Nevertheless, we can continue with the principal edge cut of
$\bar L_1\oplus_2 \bar L_3$ producing the decomposition $\{\bar
L_1, \bar L_2, \bar L_3\}$. The situation starting with
$R'=\delta_G(L_3)$ is completely symmetric to the previous one,
so we again obtain the decomposition $\{\bar L_1, \bar L_2,
\bar L_3\}$. This completes the proof of Claim~1.

\medskip

To finish the proof we employ induction on the number of
vertices to prove that every $2$-connected cubic graph $G$
admits a decomposition into cyclically $4$-edge-connected cubic
graphs that does not depend on the choice of the leading cut.
The statement vacuously holds if $G$ is cyclically
$4$-edge-connected, so we can take cyclically
$4$-edge-connected cubic graphs as the basis of induction.

For the induction step let us assume that $G$ is not cyclically
$4$-edge-connected. It means that $G$ contains at least one
small independent edge cut. First, consider the case where $G$
has exactly one small independent edge cut $R$. Let
$\{G_1,G_2\}$ be the decomposition of $G$ along~$R$. Recall
that $G_1$ and $G_2$ are uniquely determined by $R$. 
By the induction hypothesis, each $G_i$, with $i\in\{1,2\}$,
admits a decomposition
$\mathcal{D}_i$ into cyclically $4$-edge-connected cubic graphs
that does not depend on the choice of the leading cut.
Consequently, $\mathcal{D}=\mathcal{D}_1\cup\mathcal{D}_2$ is
the required unique decomposition of $G$ into cyclically 
$4$-edge-connected cubic graphs.
(In fact, it is easy to see that both $G_1$ and $G_2$ are cyclically $4$-edge-connected, so $\mathcal{D}=\{G_1,G_2\}$.)

It remains to consider the case where $G$ has at least two
distinct small independent edge cuts. Choose $R\ne R'$
arbitrarily. Claim~1 implies that $G$ admits a decomposition
$\{G_1,G_2,G_3\}$ into three smaller $2$-connected cubic graphs
that does not depend on the choice of the leading cut. By the
induction hypothesis, each $G_i$ admits a decomposition
$\mathcal{D}_i$ into cyclically $4$-edge-connected cubic graphs
that does not depend on the choice of the leading cut.
Consequently,
$\mathcal{D}=\mathcal{D}_1\cup\mathcal{D}_2\cup\mathcal{D}_3$
is a decomposition of $G$ into cyclically $4$-edge-connected
cubic graphs that does not depend on the choice of the leading
cut. In other words, $\mathcal{D}$ is the required unique
decomposition of $G$ into cyclically $4$-edge-connected graphs.
\end{proof}

The just established fact that every $2$-connected cubic graph
$G$ has a decomposition into cyclically $4$-edge-connected
factors, which is unique up to isomorphism and order of the
factors, justifies calling it the \emph{canonical
decomposition} of $G$.

\begin{remark}
The requirement that small cycle-separating edge cuts eligible
for a decomposition be independent is substantial: without this
assumption Theorem~\ref{thm:decomposition} fails. Indeed,
consider the graph $G=K_4\oplus_2 K_4$ obtained by a $2$-sum of
two copies of the complete graph $K_4$ with arbitrarily chosen
distinguished edges. If we decompose $G$ along the principal
$2$-edge-cut~$S$, the resulting decomposition will consist of
two copies of $K_4$, which cannot be further decomposed.
Alternatively, we can take a triangle $T$ in one of the
components of $G-S$. Clearly, $\delta_G(T)$ is a
cycle-separating $3$-edge-cut, and the decomposition along it
results in a copy of $K_4$ and a graph isomorphic to
$K_4\oplus_2 D_3$, where $D_3$ is the $3$-dipole. It follows
that, this time, the final decomposition consists of two copies
of $K_4$ and a copy of~$D_3$. In spite of the fact that the
uniqueness claim fails, it remains true that in both cases
every vertex of $G$ falls into precisely one decomposition
factor.
\end{remark}

We now apply Theorem~\ref{thm:decomposition} to snarks of
defect $3$.

\begin{theorem}\label{thm:decomp-def3}
Let $G$ be a $2$-connected cubic graph of defect $3$ and let
$C$ be a hexagonal core of $G$. Then $G$ admits a decomposition
$\mathcal{D}=\{G_1,\dots,G_m\}$ such that all decomposition
factors are cyclically $4$-edge-connected, except possibly
$G_1$, and the following hold:
\begin{itemize}
\item[{\rm (i)}] $G_1$ has defect $3$ and $C$ is a
    hexagonal core of $G_1$;
\item[{\rm (ii)}] each $G_i$ with $i\ge 2$ is
    $3$-edge-colourable; and
\item[{\rm (iii)}] if $G_1$ is not cyclically
    $4$-edge-connected, then it contains a triangle $T$
    such that $C$ intersects $T$ and $G_1/T$ is cyclically
    $4$-edge-connected.
\end{itemize}
Moreover, $G$ can be reconstructed from $\mathcal{D}$ by
repeated application of $2$-sums and $3$-sums.
\end{theorem}

\begin{proof}
If $G$ is cyclically $4$-edge-connected, we have nothing to
prove. In the rest of the proof we may therefore assume that
$G$ is not cyclically $4$-edge-connected, and thus the
canonical decomposition into cyclically $4$-edge-connected
graphs has at least two decomposition factors.
Theorem~\ref{thm:decomposition} states that the order in which
small independent edge cuts are used for a decomposition is
irrelevant. It follows that to produce the canonical
decomposition we may at first apply decomposition along
$2$-edge-cuts as long as possible. As a result, we obtain a
 decomposition $\mathcal{X}=\{X_1,...,X_r\}$ of
$3$-connected cubic graphs such that $G$ can be reconstructed
from them by repeated $2$-sums.

Choose an optimal $3$-array $\mathcal{M}$ for $G$, and let $C$ be its
hexagonal core. Consider an arbitrary $2$-edge-cut $R$ in $G$.
Lemma~\ref{lem:2+3-cuts}~(i) implies that $C\cap R=\emptyset$,
so decomposing $G$ along $R$ produces two graphs $K$ and $L$ 
such that, say, $K$ inherits the core. Parity Lemma implies that 
both edges of $R$ belong to the same perfect matching from $\mathcal{M}$, 
so $C$ is a core of $K$ and $L$ is $3$-edge-colourable. In
particular, $\df{K}=3$.  By applying the same consideration again 
and again we conclude that
exactly one $X_i\in\mathcal{X}$, say $X_1$, contains the
original hexagonal core $C$ of $G$ and all other graphs in
$\mathcal{X}$ are $3$-edge-colourable. Leaving $X_1$ intact, we
decompose all graphs $X_i$ with $i\ge 2$ to cyclically
$4$-edge-connected factors to produce a set $\mathcal{Y}$ of
cyclically $4$-edge-connected graphs such that
$\{X_1\}\cup\mathcal{Y}$ is a  decomposition of $G$.

At last, we process $X_1$. If $X_1$ contains no independent
$3$-edge-cut, then $\{X_1\}\cup\mathcal{Y}$ is the canonical
decomposition of $G$ satisfying the statement of this theorem.
So we may assume that $X_1$ does contain a independent
$3$-edge-cut $S$.

If $C\cap S=\emptyset$, then we are in a similar situation as
previously with $2$-edge-cuts. Thus $X_1$ decomposes along $S$
into two graphs $K_1$ and $L_1$ such that, say, $K_1$ has
defect $3$ and inherits the core $C$ while $L_1$ is
$3$-edge-colourable. We continue in this manner as long as
there are $3$-edge-cuts disjoint from~$C$, producing a
decomposition $\mathcal{X}_1=\{X_{11},\ldots,X_{1t}\}$ of
$X_1$, where, say, $X_{11}$ inherits the core $C$ and is the
only member of $\mathcal{X}_1$ with defect $3$.

The graph $X_{11}$ is $3$-connected and contains no independent
$3$-edge-cut disjoint from~$C$. If $X_{11}$ has no independent
$3$-edge-cut at all, then $\mathcal{X}_1\cup\mathcal{Y}$ is a
canonical decomposition of $G$ satisfying the statement of the
theorem, with $C$ being a hexagonal core of
\mbox{$G_1=X_{11}$}. Otherwise, there is an independent
$3$-edge-cut $S_1$ in $X_{11}$ such that $C\cap
S_1\ne\emptyset$. Let $Q_1$ and $Q_2$ be the components of
$X_{11}-S_1$. Lemma~\ref{lem:2+3-cuts}~(ii) states that the
intersection $C\cap S_1$ consists of two doubly covered edges.
Adopting the notation for $C$ and $\delta_{X_{11}}(C)$ as in
Figure~\ref{fig:core3}, we may assume that $C\cap
S_1=\{e_0,e_2\}$. One of the components of $X_{11}-S_1$, say
$Q_2$, contains the uncovered edge $e_1$ of $C$ adjacent to
both $e_0$ and $e_2$. The edges $f_1$ and $f_2$ of
$\delta_{X_{11}}(C)$ that are adjacent to $e_0$ also belong to
$Q_2$. Let $g$ be the third edge of $S_1$. The edges $f_1$ and
$f_2$ cannot be independent for otherwise $\{f_1,f_2,g\}$ would
be an independent $3$-edge-cut in $X_{11}$ disjoint from $C$.
Thus $T=(f_1e_1f_2)$ is a triangle in~$X_{11}$. Moreover, since
$X_{11}$ is $3$-connected, $\delta_{X_{11}}(T)=\{e_0,e_2,g\}$;
in particular, $Q_2=T$. Now, if we decompose $X_{11}$ along
$S_1$, we obtain the decomposition $\{\bar Q_1,\bar Q_2\}$
where $\bar Q_1$ is isomorphic to $X_{11}/T$ and $\bar  Q_2$ is
isomorphic to the complete graph~$K_4$. Clearly, both $\bar
Q_1$ and $\bar Q_2$ are cyclically $4$-edge-connected, so
$\mathcal{D}=\{\bar Q_1,\bar
Q_2\}\cup\mathcal{X}_1\cup\mathcal{Y}$ is a canonical
decomposition of~$G$. Observe that $C$ has been
inherited to $X_{11}$ from $G$ intact; in other words,
$X_{11}=\bar Q_1\oplus_3 \bar Q_2$ contains a hexagonal core,
namely $C$, intersecting a triangle. It is now sufficient to
set $G_1=X_{11}$, and the statement of the theorem follows.
This completes the~proof.
\end{proof}

\section{Perfect matching covers of cubic graphs with triangles}\label{sec:cover_triangles}
\noindent{}In this section we prove
that every snark with colouring defect $3$ that contains a hexagonal core
intersecting a triangle can have its edges covered with four
perfect matchings.

Consider a covering $\mathcal{P}$ of the edge set of a cubic
graph $G$ with four perfect matchings $P_1$, $P_2$, $P_3$, and
$P_4$. For short, we call $\mathcal{P}$ a
\emph{perfect-matching $4$-cover}. Since every vertex $v$ of
$G$ is incident with all members of $\mathcal{P}$, exactly one
edge $e$ incident with $v$ is doubly covered, that is, it
belongs to two members of $\mathcal{P}$. It is evident that the
doubly covered edges form a perfect matching of $G$.

We start with a simple lemma.

\begin{lemma}\label{lem:2+3-sum-4cover}
Let $H$ and $K$ be $2$-connected cubic graphs. If $\pi(H)=4$
and $K$ is $3$-edge-colourable, then $\pi(H\oplus_2 K)=4$ and
$\pi(H\oplus_3 K)=4$ for any choice of distinguished edges or
vertices.
\end{lemma}

\begin{proof}
We prove the lemma for $2$-sums; the proof for $3$-sums is
similar and therefore left to the reader. Let $e$ and $f$ be
the distinguished edges of $H$ and $K$ for the $2$-sum,
respectively, and let $\{g,g'\}$ be the principal $2$-edge-cut
of $H\oplus_2 K$. Take an arbitrary perfect-matching $4$-cover
$\mathcal{P}=\{P_1,P_2,P_3,P_4\}$ of $H$. Since $K$ is
$3$-edge-colourable, it contains three pairwise disjoint
perfect matchings $M_1$, $M_2$, and $M_3$ constituting the
colour classes of a 3-edge-colouring of $H$. Without loss of
generality assume that $f\in M_1$. There are two cases
depending on whether $e$ is simply or doubly covered,
respectively. If $e$ is simply covered, we may clearly assume
that $e\in P_1$. Set $M_4=M_2$, $Q_1=(P_1\cup
M_1-\{e,f\})\cup\{g,g'\}$, and for $i\in\{2,3,4\}$ set
$Q_i=P_i\cup M_i$. Then $\{Q_1,Q_2,Q_3,Q_4\}$ is a
perfect-matching $4$-cover of $G$. If $e$ is doubly covered, we
may assume that $e\in P_1\cap P_4$. Set $M_4=M_1$ and
$Q_i=(P_i\cup M_i-\{e,f\})\cup\{g,g'\}$ for $i\in\{1,4\}$ and
$Q_i=P_i\cup M_i$ for $i\in\{2,3\}$. Again,
$\{Q_1,Q_2,Q_3,Q_4\}$ is a perfect-matching $4$-cover of $G$.
\end{proof}

Next, we recall two results from \cite{KMNS-Berge}. The first
of them reveals a remarkable property of $6$-edge-cuts in
bridgeless cubic graphs

\begin{theorem}\label{thm:6cuts}
Let $G$ be a bridgeless cubic graph and let $H\subseteq G$ be a
subgraph with $|\delta_G(H)|=6$. Then $H$ has a perfect
matching, or else $H$ contains an independent set $S$ of
vertices such that
\begin{enumerate}[{\rm (i)}]
\item every vertex of $S$ is $3$-valent in $H$,
\item $\odd(H-S) =|S|+ 2$, and
\item $|\delta_G(L)| = 3$ for each component $L$ of $H-S$.
\end{enumerate}
Moreover, the graph $H/(H-S)$ obtained from $H$ by contracting
each component of $H-S$ to a vertex is bipartite.
\end{theorem}

The structure of a cubic graph $G$ with a $6$-edge-cut
$\delta_G(H)$ where $H$ is an induced subgraph with no perfect
matching is indicated in Figure~\ref{fig:Tutte}. The original
statement of Theorem~\ref{thm:6cuts} in \cite{KMNS-Berge} does
not include information about the structure of the graph
$H/(H-S)$, nevertheless, it is an immediate consequence of
Items~(i)-(iii).

\begin{figure}[h!]
\centering
\includegraphics[scale=1.1]{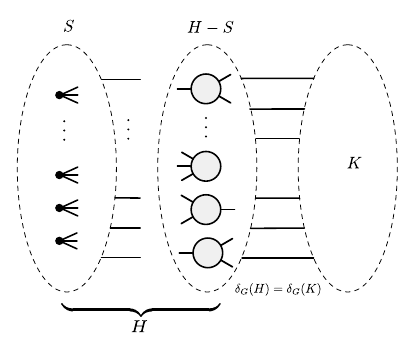}
\caption{A cubic graph $G$ containing a subgraph $H$ with
$|\delta_G(H)|=6$ that contains no perfect matching.}
\label{fig:Tutte}
\end{figure}

A cubic graph $G$ is \emph{almost bipartite} if it
is bridgeless, not bipartite, and contains two edges $e$ and
$f$ such that $G-\{e,f\}$ is a bipartite graph. The edges $e$
and $f$ are said to be \emph{surplus edges} of $G$. If $G$ is
almost bipartite, then there exists a component $K$ of $G$ such
that the surplus edges connect vertices within different
partite sets of $K$ \cite[Proposition~4.3]{KMNS-Berge}.

\begin{theorem}\cite[Theorem 4.4]{KMNS-Berge}\label{thm:albip}
Every almost bipartite cubic graph is $3$-edge-colourable.
\end{theorem}
b
Now we are ready for the main result of this section.

\begin{theorem}\label{thm:triancore}
Let $G$ be a snark with colouring defect $3$. If $G$ contains a hexagonal
core intersecting a triangle, then $\pi(G)=4$.
\end{theorem}

\begin{proof}
Let $G$ be a snark with defect $3$ containing a hexagonal core
$C$ intersected by a triangle $T$. By
Theorem~\ref{thm:decomp-def3}, $G$ admits a decomposition
$\{G_1,\dots,G_m\}$ into cyclically $4$-edge-connected cubic
graphs such that $G_1$ is a snark and each $G_i$ with $i\geq 2$
is $3$-edge colourable. Since the order of edge cuts taken in
the decomposition process is ir\-relevant, we can perform the
decomposition along $\delta_G(T)$ at the very last step,
producing a cyclically $4$-edge-connected snark $J'$ and a
graph $J''$ isomorphic to the complete graph~$K_4$ 
which contains~$T$. Lemma~\ref{lem:2+3-cuts} implies that throughout
the decomposition process the hexagonal core $C$ of $G$ is
being inherited to a smaller snark unless the cut $\delta_G(T)$ is
used. It follows that $J=J'\oplus_3 J''$ is a snark containing
a hexagonal core intersecting a triangle, both inherited from
$G$, such that $J/T$ is cyclically $4$-edge-connected. In this
situation Lemma~\ref{lem:2+3-sum-4cover} shows that it
sufficient to prove that $\pi(J)=4$. For the rest of the proof
we will therefore assume that $G=J$.

Let $\mathcal M=\{M_1,M_2,M_3\}$ be a $3$-array of perfect
matchings of $G$ whose core is~$C$, and let $\phi$ be the Fano
colouring of $G$ induced by $\mathcal{M}$.
Lemma~\ref{lem:essential}(ii) tells us that the intersection
$C\cap T$ consists of a single uncovered edge. Moreover, the
edge cut $\delta_G(T)$ comprises three pairwise independent
edges. Let $v$ be the vertex of $T$ not lying on~$C$, and let
$w$ be the neighbour of $v$ not belonging to $T$.

We claim that $w$ does not lie on $C$. If it does, then the
distance of $w$ from $T$ measured within $C$ is either $1$ or
$2$. In the former case, the edge cut $\delta_G(T)$ contains
two adjacent edges, which contradicts
Lemma~\ref{lem:essential}(ii), while in the latter case $w$ is
contained in a quadrilateral that violates
Lemma~\ref{lem:core_quadrangle}. It follows that $w$ does not
lie on $C$; the situation is illustrated in
Figure~\ref{fig:esstri}.

\begin{figure}[h!]
 \centering
 \includegraphics[scale=0.8]{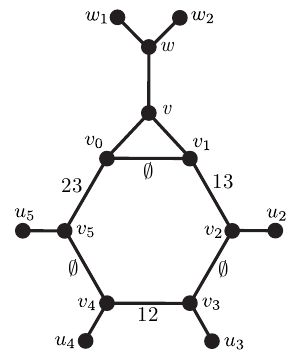}
 \caption{A hexagonal core intersecting a triangle.}
 \label{fig:esstri}
\end{figure}

We proceed to proving that $\pi(G)=4$. Assume to the contrary
that $\pi(G)\ge 5$. Take the subgraph $K$ of $G$ induced by
$C\cup T\cup\{w\}$, and set $H=G-V(K)$. Clearly, the edge cut
$R$ separating $K$ from $H$ comprises six edges, so we can
apply Theorem~\ref{thm:6cuts}. If $H$ had a perfect matching,
then the matching would extend to a perfect matching $M$ of $G$
that covers the three uncovered edges of the core and the
unique edge joining $w$ to~$T$. Consequently, $M$ together with
the elements of $\mathcal{M}$ would constitute a
perfect-matching $4$-cover of~$G$. Therefore $H$ has no perfect
matching. By Theorem~\ref{thm:6cuts}, $H$ contains an
independent set $S$ of $3$-valent vertices such that each
component $L$ of $H-S$ has $\delta_G(L)=3$. Since $G/T$ is
cyclically $4$-edge-connected, each component of $H-S$ is just
a single vertex. Therefore, $H$ is bipartite with bipartition
$\{S,V(H)-S\}$. The resulting structure of $G$ is illustrated
in  Figure~\ref{fig:Tutte}.

Let $\bar H$ denote the graph with six pendant vertices
obtained from $H\cup R$ by adding a pendant vertex to each edge
of $R$, and let $\bar K$ arise from $K\cup R$ similarly. We
analyse the $6$-edge-cut $R$ separating in $G$ the subgraph $K$
from the rest of~$G$. We enumerate the edges of $R$ as
$f_0,f_1,\ldots,f_5$, where the first two positions are
occupied by edges incident with $w$ and the positions of the
remaining four edges are taken by edges leaving $C$ listed in
the linear order determined by an orientation of $C$; see
Figure~\ref{fig:esstri}. Since the core of $\mathcal{M}$ is $C$
and $C\subseteq K$, we conclude that the graph $\bar H$ is
$3$-edge-colourable. Every $3$-edge-colouring $\sigma$ of $\bar
H$ induces a sequence of colours
$\alpha_{\sigma}=\sigma(f_0)\sigma(f_1)\ldots\sigma(f_5)$, the
\emph{colour vector} induced by $\sigma$. In general, we define
a \emph{colour vector} to be any sequence $\alpha=c_0c_1\ldots
c_5$ of colours $c_i\in\{1,2,3\}$ satisfying Parity Lemma; that
is, $c_0+\ldots+c_5=0$. By permuting the colours $1$, $2$, and
$3$ in $\alpha$ we obtain a new colour vector, nevertheless,
the difference between the two is insubstantial. Therefore each
of the six permutations of colours produces a sequence that
encodes the same colouring pattern. In order to have a
canonical representative, we choose from them the
lexicographically minimal sequence and call it the \emph{type}
of $\alpha$. A \emph{colouring type} is the type of some colour
vector~$\alpha$. The \emph{type} of a $3$-edge-colouring
$\sigma$ of $\bar H$ is the type of the induced colour vector
$\sigma(f_0)\sigma(f_1)\ldots\sigma(f_5)$. A colouring type is
said to be  \emph{admissible} for $\bar H$ if it is the type of
a certain $3$-edge-colouring of $\bar H$. We denote the set of
colouring types admissible for $\bar H$ by $\Col(\bar H)$, and
define $\Col(\bar K)$ analogously.

\medskip\noindent
Claim. \emph{Every colour vector
$\alpha\in\Col(\bar H)$ includes all three colours.}

\medskip\noindent
\emph{Proof of Claim.} Suppose that there exists a colour
vector $\gamma\in\Col(\bar H)$ that contains at most two
colours. Take any $3$-edge-colouring $\xi$ of $\bar H$ that
induces $\gamma$.  The missing colour class  forms a perfect
matching of $H$, which contradicts our assumption that $H$ has
no perfect matching. The claim is proved.

\medskip

Now we finish the proof. Parity Lemma implies that every
colouring type admissible for $\bar H$ contains all three
colours twice. It is easy to see that there exist exactly 15
colouring types containing all three colours (see
\cite[Table~1]{KMNS-Berge}). Direct verification shows that out
of them only six types are admissible for ~$\bar K$; namely,
$\Col(\bar K)=\{121233, 122133, 123123, \penalty0 123213,
123312, 123321\}$. Since $G$ is a snark, $\Col(\bar
H)\cap\Col(\bar K)=\emptyset$. It follows that $\Col(\bar H)$
is a subset of the set $\mathcal{T}$ formed by the remaining
nine types. These are displayed in Table~\ref{tbl:ctypes}.

\begin{table}[h!]
\centering
\begin{tabular}{ccccc}
\hline\hline
\rule{0pt}{11pt}$\alpha_1=112233$ & & $\alpha_4=121323$ & &
$\alpha_7=122331$\\
$\alpha_2=112323$ & & $\alpha_5=121332$ & & $\alpha_8=123132$\\$\alpha_3=112332$ & &  $\alpha_6=122313$ & & $\alpha_9=123231$\\\hline\hline
\end{tabular}
\medskip	
\caption{A set $\mathcal{T}$ of colouring types such that
$\Col(\bar H)\subseteq\mathcal{T}$} \label{tbl:ctypes}
\end{table}

We now construct a new cubic graph $G'$ from $G$ by removing
the vertices $v$, $v_0$, $v_1$, $v_3$, $v_4$, and $w$, and by
adding the path $w_2v_2v_5u_3$ and the edge $w_1u_4$, where
$w_1$ and $w_2$ are neighbours of $w$ different from $v$, see
Figure~\ref{fig:esstri_compl}.

\begin{figure}[h!]
 \centering
 \includegraphics[height=4.5cm]{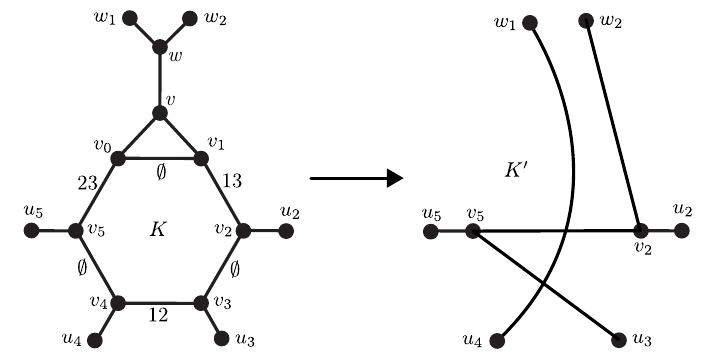}
 \caption{Constructing $G'$ from $G$.}
 \label{fig:esstri_compl}
\end{figure}

We proceed to proving that $G'$ is not $3$-edge-colourable. If
$G'$ had a $3$-edge-colouring, then such a colouring would
result from a $3$-edge-colouring $\sigma$ of $\bar H$ whose
colouring type $\alpha=c_0c_1\ldots c_5$ belongs to
$\mathcal{T}$ and satisfies all of the following conditions:
\begin{itemize}
 \item[(i)] \quad $c_0 = c_4$ \quad (the edge $f_0f_4$ is
     properly coloured),
 \item[(ii)] \quad $c_1\ne c_2$ \quad ($\sigma$ is proper
     at $s$), and
 \item[(iii)] \quad $c_3\ne c_5$ \quad ($\sigma$ is proper
     at $t$).
\end{itemize}
By checking the elements of $\mathcal{T}$ we see that all
colouring types except $\alpha_8$ violate (i), and $\alpha_8$
violates (iii). In other words, no element of $\mathcal{T}$
satisfies the conditions (i)-(iii), and therefore $G'$ admits
no $3$-edge-colouring. On the other hand, $G'$ is almost
bipartite with surplus edges $v_2v_5$ and $w_1u_4$, and by
Theorem~\ref{thm:albip} it is $3$-edge-colourable. This is the
final contradiction, which establishes the theorem.
\end{proof}

The following corollary is of particular importance.

\begin{corollary}\label{cor:esstrian}
The perfect matching index of a snark of defect $3$ with
an essential triangle equals $4$.
\end{corollary}

\begin{proof}
If a cubic graph with defect $3$ contains an essential triangle
$T$, then, by Lemma~\ref{lem:essential}~(i), every hexagonal
core of $G$ intersects $T$. The conclusion now follows from
Theorem~\ref{thm:triancore}.
\end{proof}

\section{Snarks with defect 3 and perfect matching index 5}
\label{sec:pmi5}

\noindent{}The aim of this section is to characterise snarks
with defect $3$ that cannot be covered with four perfect
matchings. Our point of departure is the following theorem
proved in \cite[Theorems~1.1 and~1.2]{KMNS-Berge}.

\begin{theorem}\label{thm:pmi-c4c}
Every snark $G$ of defect $3$ satisfies Berge's conjecture,
that is, $\pi(G)\le 5$. Moreover, if $G$ is cyclically
$4$-edge-connected, then $\pi(G)=4$ unless $G$ is the Petersen
graph, which has $\pi(Pg)=5$.
\end{theorem}

What remains to be done here is to provide a characterisation
of snarks with defect~$3$ and perfect matching index~$5$ whose
cyclic connectivity equals $2$ or $3$.

\medskip

The next lemma follows from straightforward parity arguments.

\begin{lemma}\label{lem:pmc-2+3cuts}
Let $G$ a $2$-connected cubic graph endowed with a
perfect-matching $4$-cover $\mathcal{P}=\{P_1,P_2,P_2,P_4\}$.
The following statements hold.
\begin{itemize}
\item[{\rm (i)}] If $R$ is a $2$-edge-cut of $G$, then
    either both edges of $R$ are simply covered or both are
    doubly covered. Moreover, they belong to the same
    members of $\mathcal{P}$ (either one or two).

\item[{\rm (ii)}] If $R$ is a $3$-edge-cut of $G$, then
    either one or all three edges are doubly covered. In
    the latter case, there is a perfect matching
    $P_j\in\mathcal{P}$ such that $R\subseteq P_j$.
\end{itemize}
\end{lemma}

We continue with a series of five lemmas preparing our main
result.

\begin{lemma}\label{lem:3cut4cover}
Let $G=H\oplus_3 K$ be a $3$-sum of $2$-connected cubic graphs
$H$ and $K$ such that $\pi(G)=4$. If $\pi(H)\geq 5$ or
$\pi(K)\geq 5$, then each edge of the principal $3$-edge-cut is
doubly covered.
\end{lemma}

\begin{proof}
Let $H'$ and $K'$ be the graphs obtained from $H$ and $K$ by
deleting their respective distinguished vertices, and let
$\mathcal{P}$ a perfect-matching $4$-cover of $G$. By
Lemma~\ref{lem:pmc-2+3cuts}~(ii), either one or all three edges
of the principal $3$-edge-cut are doubly covered. If the former
occurs, then $\mathcal{P}$ induces a perfect-matching $4$-cover
of each of $G/K'$ and $G/H'$. Since $G/K'$ and $G/H'$ are
isomorphic to $H$ and $K$, respectively, we conclude that
$\pi(H)=4=\pi(K)$. Thus if $\pi(H)\ge 5$ or $\pi(K)\ge 5$, then
all three edges of the principal edge-cut must be doubly
covered.
\end{proof}

\begin{lemma}\label{lem:2sum}
Let $G=H\oplus_2 K$ be a $2$-sum of $2$-connected cubic graphs
$H$ and $K$, where $K$ is $3$-edge-colourable. Then $\pi(G)\ge
5$ if and only if $\pi(H)\ge 5$.
\end{lemma}
	
\begin{proof}
By Lemma~\ref{lem:2+3-sum-4cover}, it suffices to prove that if
$\pi(H\oplus_2 K)\leq 4$, then $\pi(H)\leq 4$. Assume that $G$
has a perfect-matching $4$-cover
$\mathcal{P}=\{P_1,P_2,P_3,P_4\}$. By
Lemma~\ref{lem:pmc-2+3cuts}~(i), there exist
$i,j\in\{1,2,3,4\}$ such that such that both edges of the
principal $2$-edge-cut belong to $P_i\cap P_j$ (possibly
$i=j$). A perfect-matching $4$-cover
$\mathcal{P}'=\{P_1',P_2',P_3',P_4''\}$ of $H$ can now be
defined in a natural way: for $k\notin \{i,j\}$ we set
$P_k'=P_k\cap E(H)$ and for $k\in\{i,j\}$ we set $P_k'=P_k\cap
E(H)\cup\{e\}$, where $e$ is the distinguished edge of $H$ for
the $2$-sum.
\end{proof}

A vertex $v$ of a snark $G$ will be called an \emph{apex} if
the graph $G^v$ obtained by inflating $v$ into a triangle has
$\pi(G^v)=4$. This concept is useful especially if $\pi(G)\ge
5$. It is well known (see \cite[p.~146]{EM}) that every vertex
of the Petersen graph is an apex. By contrast, the windmill
snark of order~34 (see \cite[Figure~7]{EM}) has exactly three
vertices that are not apexes, the neighbours of the central
vertex.

\begin{lemma}\label{lem:apex}
In a snark $G$ with $\pi(G)\ge 5$, no apex is incident with a
pair of parallel edges.
\end{lemma}

\begin{proof}
If a vertex $v$ of a snark $G$ is incident with a pair of
parallel edges, then $G^v$ has a $2$-edge-cut whose removal
yields a component consisting of two triangles sharing an edge.
Using Lemma~\ref{lem:pmc-2+3cuts}~(i) it is easy to see that
every perfect-matching $4$-cover of $G^v$ gives rise to a
perfect-matching $4$-cover of $G$. Hence, $\pi(G)=4$ and $v$ is
not an apex.
\end{proof}

\begin{lemma}\label{lem:pmi4sum3gen}
Let $G=H\oplus K$ be a $2$-sum or a $3$-sum of $2$-connected
cubic graphs $H$ and $K$ where $H$ is a snark with $\pi(H)\ge
5$ and $K$ is $3$-edge-colourable. If $u$ is an apex of $H$
different from the distinguished vertex for the $3$-sum, then
$u$ is an apex of $G$.
\end{lemma}

\begin{proof}
If $u$ is an apex of $H$, then $\pi(H^u)=4$. The graph $K$ is
$3$-edge-colourable and $u$ is not a distinguished vertex for
the sum, so $\pi(H^u\oplus K)=4$ by
Lemma~\ref{lem:2+3-sum-4cover}. Since $\pi(H^u\oplus K)=G^u$,
we conclude that $u$ is an apex of $G$.
\end{proof}

A $2$-connected cubic graph $K$ will be called
\emph{quasi-bipartite} if it contains an independent set of
vertices $U$ with $|U|\ge 2$ such that the graph $\tilde
K=K/(K-U)$ obtained by the contraction of each component of
$K-U$ to a vertex is a bipartite cubic graph with partite sets
$U$ and $V(\tilde K)-U$; the graph $\tilde K$ may not be
simple. The independent set $U$ is a \emph{bipartising set} of
$K$ and the set of components of $K-U$ is a \emph{quasi-partite
set} of~$K$. Clearly, the quasi-partite set and the bipartising
set have the same size.

\begin{lemma}\label{lem:claw}
Let $G$ be a bridgeless cubic graph and let $H\subseteq G$ be
an induced subgraph with $|\delta_G(H)|=6$. If $H$ has no
perfect matching, then every cubic graph obtained from $H$ by
adding a $3$-star $Y$ to $H$ and attaching the edges of
$\delta_G(H)$ to the leaves of $Y$ is quasi-bipartite.
Moreover, the central vertex of $Y$ forms a trivial component
of the corresponding quasi-partite set of $K$.
\end{lemma}

\begin{proof}
Take a $3$-star $Y$ disjoint from $H$ with leaves $w_1$, $w_2$,
$w_3$, and central vertex $z$. Let $H^{\sharp}$ be a cubic
graph obtained by adding $Y$ to $H$ and attaching the edges of
$\delta_G(H)$ to the leaves of~$Y$. From
Theorem~\ref{thm:6cuts} we know that $H$ contains an
independent set $S$ of $3$-valent vertices such that each
component $L$ of $H-S$ has $\delta_G(L)=3$ and the number of
components of $H-S$ is $|S|+2$. Let $\{L_1,\ldots,L_m\}$ be the
set of components of $H-S$; so $|S|=m-2$. Set
$U=S\cup\{w_1,w_2,w_3\}$. Then $U$ is an independent subset of
$3$-valent vertices of $H^\sharp$ such that $L_1,\ldots,L_m$
and $\{z\}$ are the components of $H^{\sharp}-U$. Due to the
construction of $H^\sharp$, each component $L$ of $H^{\sharp}
-U$ has $\delta_{H^\sharp}(L)=3$, so
$H^{\sharp}/(H^{\sharp}-U)$ is a bipartite cubic graph. Hence,
$H^{\sharp}$ is quasi-bipartite.
\end{proof}

The following two theorems are the last ingredients needed for
our main result.

\begin{theorem}\label{thm:3-sumpmi5}
Let $H$ and $K$ be $2$-connected cubic graphs, where $H$ is a
snark with $\pi(H)\ge 5$ and $K$ is $3$-edge-colourable
quasi-bipartite graph. If $G=H\oplus_3 K$ is a $3$-sum with
distinguished vertices $u$ and $v$, where
$v$ forms a trivial component of the quasi-partite set of $K$,
then the following hold:
\begin{itemize}
\item[{\rm (i)}] $\pi(G)\ge 5$.

\item[{\rm (ii)}] $G$ is quasi-bipartite with bipartising
    set inherited from $K$, and $H-u$ is an element of the
    quasi-partite set of $G$ replacing $v$.
\end{itemize}
\end{theorem}

\begin{proof}
To prove (i) suppose to the contrary that $G$ has a
perfect-matching $4$-cover, and let $|U|=n$. Since $\pi(H)\geq
5$, Lemma~\ref{lem:3cut4cover} tells us that all edges of the
principal $3$-edge-cut $R$ of $G$ are doubly covered. Recall
that the set $M$ of all doubly covered edges of $G$ forms a
perfect matching. Since $R=\delta_G(H-u)$ and $R\subseteq M$,
we conclude that $H-u$ adds three edges of $M$ to the matching
$M\cap\delta_G(U)$. At the same time, each other component of
$G-U$ contributes to $M\cap\delta_G(U)$ by at least one edge.
The number of components of $G-U$ equals $n$, therefore
$|M\cap\delta_G(U)|\ge n+2$. On the other hand, $U$ is an
independent set, so $|M\cap\delta_G(U)|=n$. This contradiction
proves (i). Statement (ii) is obvious.
\end{proof}

The next theorem is a partial converse of
Theorem~\ref{thm:3-sumpmi5}.

\begin{theorem}\label{thm:3-sumpmi5-converse}
Let $G=H\oplus_3 K$ be a proper $3$-sum of $2$-connected cubic
graphs where $\pi(H)\ge 5$, $K$ is $3$-edge-colourable, and the
distinguished vertex of $H$ is an apex. If $\pi(G)\ge 5$, then
$K$ is quasi-bipartite and its distinguished vertex forms a
trivial component of the quasi-partite set of $K$.
\end{theorem}

\begin{proof}
The principal $3$-edge-cut $R=\{e_1,e_2,e_3\}$ of $G$ consists
of three independent edges because the $3$-sum is proper. In
particular, the neighbours $v_1$, $v_2$, and $v_3$ of the
distinguished vertex $v$ of $K$ are pairwise distinct. Consider
the graph $K^{\flat}=K-\{v,v_1,v_2,v_3\}$ obtained by removing
from~$K$ the $3$-star induced by $v$ and its neighbours.
Observe that $|\delta_G(K^{\flat})|=6$, so we can apply
Theorem~\ref{thm:6cuts}. We claim that $K^{\flat}$ has no
perfect matching. Suppose to the contrary that $K^{\flat}$ has
a perfect matching which we denote by $M_4$. Since
$\pi(H^u)=4$, there is a perfect-matching $4$-cover
$\mathcal{P}=\{P_1,P_2,P_3,P_4\}$ of $H^u$. Note that
$\delta_{H^u}(H-u)=\delta_G(H-u)=R$. By
Lemma~\ref{lem:3cut4cover}, all the edges of $R$ are doubly
covered because $\pi(H)\ge 5$. Moreover,
Lemma~\ref{lem:pmc-2+3cuts}~(ii) tells us that there exist an
index $j\in\{1,2,3,4\}$ such that $R\subseteq P_j$. Without
loss of generality we may assume that $j=4$ and $e_i\in P_i\cap
P_4$ for each $i\in\{1,2,3\}$. By our assumption, $K$ has a
$3$-edge-colouring; let $M_1$, $M_2$, and $M_3$ be its colour
classes such that $e_i\in M_i$ for each $i\in\{1,2,3\}$. If we
set $Q_i=P_i\cup M_i$ for $i\in\{1,2,3,4\}$, we can see that
$\{Q_1,Q_2,Q_3,Q_4\}$ is a perfect-matching $4$-cover of $G$,
and we have a contradiction. Therefore $K^{\flat}$ has no
perfect matching, and Lemma~\ref{lem:claw} now implies that $K$
is quasi-bipartite with $v$ forming a trivial component of the
quasi-partite set of $K$.
\end{proof}

The following theorem characterising cubic graphs of colouring
defect $3$ whose perfect matching index equals $5$ is one of
the main results of this paper. For its formulation it is
convenient to regard a $2$-sum $H\oplus_2 K$ or a $3$-sum
$H\oplus_3 K$ as a modification of the graph $H$ by a graph
$K$. To be precise, consider a $2$-connected cubic graph $H$
and vertex $u$ and an edge $e$ of $G$. We will say that a cubic
graph $G$ arises from $H$ by an \emph{insertion} of an
edge-deleted cubic graph into $e$ if there exists a
$2$-connected cubic graph $K$ and an edge $f$ of $K$ such that
$G=H\oplus_2 K$ with distinguished edges $e$ and $f$,
respectively. Specifically, it is the insertion of the
edge-deleted graph $K-f$ into $e$. We will further say that $G$
arises from $H$ by a \emph{substitution} of the vertex $u$ with
a vertex-deleted cubic graph if there exists a $2$-connected
cubic graph $K$ and a vertex $v$ of $K$ such that $G=H\oplus_3
K$ is a $3$-sum with distinguished vertices $u$ and $v$,
respectively. Specifically, it is a substitution of $u$ with
$K-v$. For example, inflating a vertex $u$ in a cubic graph to
a triangle is nothing but a substitution of $u$ with $K_4-v$,
where $K_4$ is the complete graph on four vertices.

We conclude this series of definitions with the following. A
substitution of a vertex $u$ of a cubic graph $H$ with a
vertex-deleted quasi-bipartite cubic graph $K-v$ is said to be
\emph{correct} if $K$ has a bipartising set $U$ such that $v$
constitutes a component of the graph $K-U$; that is, if $\{v\}$
belongs to the corresponding quasi-partite set of $K$.

In our characterisation theorem, insertions and substitutions
will be performed with $3$-edge-colourable graphs. In this
context it may be useful to keep in mind that an edge-deleted
or a vertex-deleted cubic graph is $3$-edge-colourable if and
only the original cubic graph is.

\begin{theorem}\label{thm:main}
Every $2$-connected cubic graph with colouring defect $3$ and
perfect matching index $5$ is either isomorphic to the Petersen
graph $Pg$ or arises from the Petersen graph with a fixed
$6$-cycle $C$ by combining the following two operations:
\begin{itemize}
\item[{\rm (O1)}] insertion of connected
    $3$-edge-colourable edge-deleted cubic graphs into any
    number of edges from $E(Pg)-E(C)$; and
\item[{\rm (O2)}] correct substitution of any number of
    vertices from $V(Pg)-V(C)$ with connected
    vertex-deleted $3$-edge-colourable quasi-bipartite
    cubic graphs.
\end{itemize}
Conversely, every cubic graph constructed from the Petersen
graph by means of operations {\rm (O1)} and {\rm (O2)} has
defect $3$ and perfect matching index~$5$. Moreover, the
$6$-cycle $C$ constitutes a hexagonal core of the resulting
graph.
\end{theorem}

\begin{proof}
$(\Leftarrow)$ We show that every graph $G$ constructed from
$Pg$ with a fixed $6$-cycle $C$ by means of operations (O1) and
(O2) has defect $3$ and perfect matching index $5$. As regards
its perfect matching index, Theorem~\ref{thm:pmi-c4c} implies
that it is enough to show that $\pi(G)\ge 5$. Thus it
sufficient to prove the following by induction on $t$:

\medskip\noindent
Claim 1. \emph{Let $G_t$ be the graph obtained from the
Petersen graph with a fixed $6$-cycle $C$ by applying $t$
operations {\rm (O1)} and {\rm (O2)} in any order. Then
$\pi(G_t)\ge 5$ and $G_t$ has a $3$-array $\mathcal{M}_t$ whose
core is $C$.}

\medskip\noindent
\emph{Proof of Claim 1.} Before starting the induction
procedure it may be useful to realise that $0\le t\le 13$. For
$t=0$ the statement is obviously true. Assume that the
statement is true for some $t\ge 0$, and consider the graph
$G_{t+1}$.

First assume that $G_{t+1}$ arises from $G_t$ by applying an
edge-insertion into an edge $e\in E(Pg)-E(C)$. Thus there
exists a connected $3$-edge-colourable graph $K$ and an edge
$k$ of $K$ such that $G_{t+1}=G_t\oplus_2 K$ with $e$ and $k$
being the distinguished edges for the $2$-sum, respectively.
Let $R$ be the principal $2$-edge-cut of the $2$-sum. Due to
the its choice, $e$ is simply covered with respect to the
$3$-array $\mathcal{M}_t=\{M_1^t,M_2^t,M_3^t\}$. Without loss
of generality we may assume that $e\in M_1^t$. Let $N_1$, $N_2$
and $N_3$ be the colour classes of an arbitrary
$3$-edge-colouring of~$K$; we can clearly assume that $k\in
N_1$. Set $M_1^{t+1}=(M_1^t-\{e\})\cup (N_1-\{k\})\cup R$ and
$M_i^{t+1}=M_i^t\cup N_i$ for $i\in\{2,3\}$. Obviously,
$\mathcal{M}_{t+1}=\{M_1^{t+1},M_2^{t+1},M_3^{t+1}\}$ a
$3$-array for $G_{t+1}$ with $C$ being its hexagonal core. The
fact that $\pi(G_{t+1})\ge 5$ follows immediately from
Lemma~\ref{lem:2sum}.

Next assume that $G_{t+1}$ arises from $G_t$ by substituting a
vertex $u\in V(Pg)-V(C)$. Now, there is a connected
$3$-edge-colourable quasi-bipartite cubic graph $K$ and a
vertex $v$ of $K$ such that $G_{t+1}=G_t\oplus_3 K$, with $u$
and $v$ being the distinguished vertices for the $3$-sum.
Theorem~\ref{thm:3-sumpmi5} implies that $\pi(G)\ge 5$.
Moreover, since $u$ does not lie on $C$, each edge incident
with $u$ is simply covered with respect to the $3$-array
$\mathcal{M}_t$. Since $K$ is $3$-edge-colourable, each perfect
matching $M_i^t\in\mathcal{M}_t$ can be extended to a perfect
matching $M_i^{t+1}$ of $G_{t+1}$ in such a way that the core
of the $3$-array
$\mathcal{M}_{t+1}=\{M_1^{t+1},M_2^{t+1},M_3^{t+1}\}$ is
again~$C$.

$(\Rightarrow)$ Let $G$ be a $2$-connected cubic graph with
$\df{G}=3$ and $\pi(G)=5$ different from the Petersen graph
$Pg$. Let $C$ be an arbitrary hexagonal core of $G$. We intend
to show that $G$ can be obtained from $Pg$ by a series of
operations (O1) and (O2) leaving $C$ intact.  We start by
applying Theorem~\ref{thm:decomp-def3}. It states that $G$
admits a decomposition $\mathcal{G}=\{G_1,\ldots,G_m\}$ such
that $G_1$ has defect $3$ and inherits $C$ as a hexagonal core,
all other decomposition factors are $3$-edge-colourable and
cyclically $4$-edge-connected, while $G_1$ is either cyclically
$4$-edge-connected or contains a triangle $T$ intersected by
$C$ such that $G_1/T$ is cyclically $4$-edge-connected.
Moreover, $G$ can be reconstructed from $\mathcal{G}$ by a
repeated application of $2$-sums and proper $3$-sums. We first
prove that that $\pi(G_1)\ge 5$. Suppose to the contrary that
$\pi(G)=4$. Since $G$ can be obtained from $\mathcal{G}$ by
applying $2$-sums and $3$-sums, we can repeatedly apply
Lemma~\ref{lem:2+3-sum-4cover} to deduce that $\pi(G)=4$, which
contradicts the assumption. Thus $\pi(G_1)\ge 5$. 
We further claim that $G_1$ is
cyclically $4$-edge-connected, implying that $\mathcal{G}$ is
the canonical decomposition of $G$. If $G_1$ is not cyclically
$4$-edge-connected, then, by Theorem~\ref{thm:decomp-def3}, it
contains a triangle $T$ intersecting~$C$, which is s a
hexagonal core of $G_1$. Theorem~\ref{thm:triancore} now yields
that $\pi(G_1)=4$, which again is a contradiction. Therefore
$G_1$ is a cyclically $4$-edge-connected cubic graph with
$\df{G_1}=3$ and $\pi(G_1)\ge 5$; according to
Theorem~\ref{thm:pmi-c4c}, $G_1$ is the Petersen graph $Pg$.

To show that $G$ can be obtained from the Petersen graph by
using operations (O1) and (O2) we build $G$ from $G_1=Pg$ by
using the canonical decomposition $\mathcal{G}$ of $G$. Recall
that $C$ is a hexagonal core of $G_1$. Further recall that
certain edges and vertices in $G_1$ are distinguished for
$2$-sums and $3$-sums applied during the reconstruction of $G$
from~$\mathcal{G}$. The same holds for the remaining
decomposition factors as well. The uniqueness of the
decomposition implies that during the reconstruction process
$2$-sums and proper $3$-sums can be used in any order. Thus we
can begin with $2$-sums and $3$-sums that do not use
distinguished edges or vertices from~$G_1$. Applying $2$-sums
and $3$-sums to all such distinguished edges and vertices
results in a decomposition $\mathcal{Q}=\{Q_1,\ldots, Q_r\}$
of~$G$, where $Q_1=G_1$ is the Petersen graph and each $Q_i$
with $i\ge 2$ is a connected $3$-edge-colourable graph. Now we
reconstruct $G$ from $\mathcal{Q}$. It is clear that each
$2$-sum or $3$-sum applied in this reconstruction involves an
edge or a vertex of $Q_1=Pg$. In fact, it follows from the
definition of the decomposition of a cubic graph along a small
independent edge cut that each $Q_i$ with $i\ge 2$ contains
exactly one distinguished element whose counterpart lies
in~$Q_1$. The distinguished vertices and edges lying in $Q_1$
must occur outside $C$, because $C$ has been left intact during
the entire decomposition process. In particular, $r\le 14$.

We reorder $\mathcal{Q}$ in such a way that the graphs $Q_2,
\ldots, Q_p$ have distinguished edges and the graphs $Q_{p+1},
\ldots, Q_r$ have distinguished vertices. We can now construct
$G$ inductively by setting $H_1=Q_1$ and $H_{i+1}=H_i\oplus_2
Q_{i+1}$ for $i\in\{1,\ldots, p-1\}$ and $H_{i+1}=H_i\oplus_3
Q_{i+1}$ for $i\in\{p, \ldots, r-1\}$. Consequently, $G=H_r$.
Note that each $H_i$, with $p\le i\le r-1$, has a distinguished
vertex, denoted by $u_i$, which is an original vertex of
$Q_1=Pg$.

What remains to be proved is the following statement.

\medskip\noindent
Claim 2. \emph{For every $i\ge p+1$ the graph
$Q_i\in\mathcal{Q}$ is quasi-bipartite and each $H_i$ arises
from $H_{i-1}$ by a correct substitution of its distinguished
vertex $u_i$.}

\medskip\noindent
\emph{Proof of Claim 2.} Our aim is to apply
Theorem~\ref{thm:3-sumpmi5-converse}. To this end we first
prove that every graph $H_i$ with $i\ge p$ has $\pi(H_i)\ge 5$
and $u_i$ is an apex of $H_i$. Recall that each $Q_i$ with
$2\le i\le p$ is $3$-edge-colourable, so we can repeatedly
apply Lemma~\ref{lem:2sum} to conclude that $\pi(H_i)\ge 5$
whenever $1\le i \le p$. In particular, $\pi(H_p)\ge 5$. To
prove that $\pi(H_i)\ge 5$ also for each $i\ge p+1$ we start
with the fact that $G=H_r=H_{r-1}\oplus_3 Q_r$ where
$\pi(H_r)=\pi(G)=5$ and $Q_r$ is $3$-edge-colourable. It
follows that $\pi(H_{r-1})\ge 5$, because otherwise
Lemma~\ref{lem:2+3-sum-4cover} would imply that
$\pi(G)=\pi(H_{r-1}\oplus_3 Q_r)=4$, contrary to the
assumption. If we use the same argument inductively with
decreasing index, we can conclude that every graph $H_i$ with
$p+1\le i\le r$ has $\pi(H_i)\ge 5$.

Next we prove that each distinguished vertex $u_i$ of $H_i$
with $p\le i\le r-1$ is an apex. Recall that each $H_i$ arises
from $Pg$ by a repeated application of edge-insertion or
vertex-substitution. Since every vertex in the Petersen graph
is an apex, repeated use of Lemma~\ref{lem:pmi4sum3gen} implies
that each $u_i$ for $i\geq p$ is an apex of $H_i$.

Finally we prove that each $Q_i$ with $i\ge p+1$ is
quasi-bipartite. We already know that $\pi(H_i)\ge 5$ and that
$u_i$ is apex for each $i\ge p$. Since $H_{i+1}=H_i\oplus_3
Q_{i+1}$ whenever $p\le i\le r-1$, from
Theorem~\ref{thm:3-sumpmi5-converse} we infer that $Q_{i+1}$ is
quasi-bipartite and $H_{i+1}$ arises from $H_i$ by a correct
substitution of $u_i$. This completes the proof of Claim~2 and
also the proof of the theorem.
\end{proof}

\section{Short cycle covers of cubic graphs}
\noindent{}The purpose of this section is to apply results of
preceding sections to proving results related to the
$7/5$-conjecture. Recall that a \emph{cycle cover} of $G$ is a
collection $\mathcal{C}$ of cycles such that each edge of~$G$
belongs to at least one member of $\mathcal{C}$. The
\emph{length} of a cycle cover~$\mathcal{C}$, denoted by
$\ell(\mathcal C)$, is the sum of lengths of all its cycles.
The length of a shortest cycle cover of $G$ is denoted by
$\scc(G)$. The quantity $\scc(G)/|E(G)|$ is the \emph{cycle
covering ratio} of $G$.

Given a  cycle cover $\mathcal{C}$ of $G$, we define the
\emph{weight} of an edge $e$ with respect to $\mathcal{C}$,
denoted by $w_{\mathcal{C}}(e)$ (or simply $w(e)$, if no
confusion arises), to be the number of cycles of $\mathcal{C}$
that contain $e$. The \emph{weight} of a vertex $v$, denoted by
$w(v)$, is the sum of the weights of all edges incident with
$v$. The weight of every vertex is obviously even.

As previously mentioned, the length of every cycle cover of any
bridgeless cubic graph is at least $4/3\cdot m$, where $m$ is
the number of edges. This value is known \cite{S2} to be
reached by cubic graphs with perfect matching index at most
$4$. For the Petersen graph we have $\scc(Pg)=4/3\cdot m+1$ as
indicated in Figure~\ref{fig:Pg-cycle-cover}. As we shall see
next, the value $4/3\cdot m+1$ is shared by all snarks of
defect $3$ and perfect matching index $5$.

\begin{figure}[h!]
\centering
\includegraphics[scale=1.4]{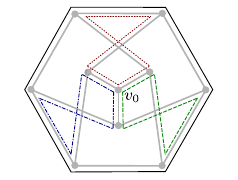}
\caption{A shortest cycle cover for the Petersen graph.}
\label{fig:Pg-cycle-cover}
\end{figure}

The following lemma will be useful.

\begin{lemma}\label{lem:plus1}
Let $G$ be a cubic graph with $m$ edges. A cycle cover of $G$
has length $4/3\cdot m+1$ if and only all vertices have weight
$4$, except for one, which has weight $6$.
\end{lemma}

\begin{proof}
Assume that $G$ has a cycle cover $\mathcal C$ such that
exactly one vertex of $G$ has weight $6$ and all other vertices
have weight $4$. Then
$$\ell(\mathcal{C})=\frac{1}{2}\sum  w(v)
=(4|V(G)|+2)/2=2|V(G)|+1=4/3\cdot |E(G)|+1,$$ as claimed. The
converse follows by reading this calculation backwards.
\end{proof}

Here is the main result of this section.

\begin{theorem}\label{thm:scc-def3}
Every snark of defect $3$ has a cycle cover of length at
most $4/3\cdot m+1$, where $m$ is the number of edges.
\end{theorem}

\begin{proof}
If $\pi(G)\le4$, then $\scc(G)=4/3\cdot|E(G)|$ according to
\cite[Theorem~3.1]{S2}. We may therefore assume that
$\pi(G)\ge5$. By Theorem~\ref{thm:main}, $G$ arises from the
Petersen graph $Pg$ with a fixed $6$-cycle $C$ by a series of
operations (O1) and (O2), that is, edge-insertions involving
connected $3$-edge-colourable cubic graphs and correct
vertex-substitutions involving connected $3$-edge-colourable
quasi-bipartite cubic graphs.

Lemma~\ref{lem:plus1} implies that it is sufficient to
construct a cycle cover of $G$ such that exactly one vertex of
$G$ has weight $6$ and all other vertices have weight $4$. We
construct the required cycle cover starting from a cycle cover
$\mathcal{P}$ of the Petersen graph which consists of three
pentagons sharing the same vertex $v_0$ plus a hexagon induced
by all the vertices at distance~$2$ from $v_0$, see
Figure~\ref{fig:Pg-cycle-cover}. In particular, $v_0$ is the
only vertex of $Pg$ of weight~$6$ with respect to
$\mathcal{P}$, all other vertices having weight $4$. Since $Pg$
is vertex-transitive, we may choose $\mathcal{P}$ in such a way
that $v_0$ lies on our specified 6-cycle $C$. This will
guarantee that during the construction of $G$ from $Pg$ the
vertex $v_0$ remains intact and the weight of $v_0$ does not
change.

We proceed by induction on the number of edge-insertions (O1)
and vertex-substitutions (O2) performed on $Pg$ to construct
$G$. For $t\ge 0$, let $G_t$ denote a graph obtained from the
Petersen graph by a sequence of $t$ edge-insertions and
vertex-substitutions of Theorem~\ref{thm:main}. We intend to
prove that each $G_t$ has a cycle cover $\mathcal{P}_t$ under
which all vertices of $G_t$ have weight $4$ except for $v_0$,
which has weight $6$.

For $t=0$ we set $\mathcal{P}_0=\mathcal{P}$, which verifies
the induction basis. For the induction step assume that the
required cycle cover $\mathcal{P}_t$ exists for some $t\ge 0$.
If all the edges from $E(Pg)-E(C)$ and all the vertices of
$V(Pg)-V(C)$ have already been used by operations (O1) and
(O2), then we have nothing to prove. Otherwise, there is an
edge in $E(Pg)-E(C)$ or a vertex $V(Pg)-V(C)$ that have not
been used yet.

First assume that $G_{t+1}$ arises from $G_t$ by applying an
edge-insertion into an edge $e\in E(Pg)-E(C)$. Thus there is a
connected $3$-edge-colourable graph $H$ and an edge $h$ of $H$
such that $G_{t+1}=G_t\oplus_2 H$ with $e$ and $h$ being the
distinguished edges for the $2$-sum, respectively. Take an
arbitrary $3$-edge-colouring $\psi$ of $H$ and let $C_{i,j}$,
with $i,j\in\{1,2,3\}$, denote the bi-coloured $2$-factor
consisting of the edges coloured $i$ or $j$. As already
mentioned, any two distinct bi-coloured $2$-factors constitute
a cycle cover of $H$ with every vertex of weight~$4$.

Without loss of generality we may assume that $\psi(e)=1$.
Since $e$ has been inherited from $Pg$ to $G_t$, its weight
$w_t(e)$ under $\mathcal{P}_t$ is the same as that under
$\mathcal{P}$, which may be $1$ or $2$. If $w_t(e)=1$, we merge
$\mathcal{P}_t$ with $\{C_{1,2}, C_{2,3}\}$ into a cycle cover
$\mathcal{P}_{t+1}$ of $G_{t+1}$ in an obvious way; if
$w_t(e)=2$, we use $\{C_{1,2}, C_{1,3}\}$ instead. It is easy
to see that in both cases the cover $\mathcal{P}_{t+1}$ has the
required property.

Next assume that $G_{t+1}$ arises from $G_t$ by substituting a
vertex $u\in V(Pg)-V(C)$. Now, there is a connected
$3$-edge-colourable quasi-bipartite cubic graph $K$ and a
vertex $v$ of $K$ such that $G_{t+1}=G_t\oplus_3 K$, with $u$
and $v$ being the distinguished vertices for the $3$-sum. Since
the weight of $u$ under the cover $\mathcal{P}_t$ is $4$,
exactly one of the edges incident with $u$ in $G_t$ has weight
$2$, say $e_1$. Let $e_1'$ be the edge of $K$ incident with $v$
that is glued with $e_1$ to form an edge of the principal
$3$-edge-cut of $G_t\oplus_3 K$. Consider an arbitrary
$3$-edge-colouring $\sigma$ of $K$. If $\sigma(e_1')=1$, we
construct a cycle cover $\mathcal{P}_{t+1}$ of $G_{t+1}$ by
merging $\mathcal{P}_t$ with the cycle cover $\{C_{1,2},
C_{1,3}\}$ of $K$. If $\sigma(e_1')\ne 1$ we apply a suitable
permutation of colours and proceed as before. Clearly, the
vertex $v_0$ is the only vertex of weight $6$ under the cycle
cover $\mathcal{P}_{t+1}$. This concludes the induction step.

By Theorem~\ref{thm:main}, $G=G_t$ for some $t\ge 0$,
and in all the considered cases $v_0$ remains the only vertex
of weight $6$, with all other vertices being of weight $4$. By
Lemma~\ref{lem:plus1}, $G$ has a cycle cover of length
$4/3\cdot m+1$, as claimed. This completes the proof.
\end{proof}

We have arrived at our final result, which shows that every
cubic graph satisfying a simple assumption independent of its
colouring defect admits a cycle cover of length at most $4/3 \cdot m +1$.
Recall that a pentagon $D$ in a snark $G$ is heavy if $D$
contains an edge $x$ such that $G\sim x$ is
$3$-edge-colourable.

\begin{theorem}\label{thm:pentagons}
If a snark $G$ contains a heavy pentagon, then $\scc(G)\le
4/3\cdot|E(G)|+1$.
\end{theorem}

\begin{proof}
Let $G$ be a snark containing a heavy pentagon $D$ of $G$ and
let $v$ be a vertex of $D$. Form the graph $G^{v}$ by inflating
$v$ to a triangle $T$. Theorem~\ref{thm:inflation} tells us
that $\df{G^{v}}=3$ and Theorem~\ref{thm:triancore} further
states that $\pi(G^{v})=4$. By Theorem~3.1 in \cite{S2},
$G^{v}$ has a cycle cover $\mathcal{C}$ of length
$\ell(\mathcal{C})=4/3\cdot|E(G^{v})|$. One can easily see that
in a cycle cover of this length each vertex of $G^{v}$ has
weight equal to $4$.  Consequently, at most one of the edges of
$T$ has weight $2$. Observe that to cover the edges leaving $T$
one needs at least two distinct members of $\mathcal{C}$. It
follows that $T$ cannot be a part of any cycle from
$\mathcal{C}$, for otherwise $T$ would have a vertex of weight
at least $6$. Now, if we contract $T$ back to $v$, every cycle
from $\mathcal{C}$ transforms to a cycle of $G$. In other
words, $\mathcal{C}$ induces a cycle cover $\mathcal{D}$ of
$G$. Since the contribution of $T$ to $\ell(\mathcal{C})$ is at
least $3$, we conclude that
$\ell(\mathcal{D})\le\ell(\mathcal{C})-3$. Consequently,
$$\scc(G)\le\ell(\mathcal{D})\le\ell(\mathcal{C})-3=
4/3\cdot(|E(G)|+3)-3=4/3\cdot|E(G)|+1,$$
as claimed. \end{proof}

Note that the Petersen graph satisfies the assumptions of
Theorem~\ref{thm:pentagons} and reaches the upper bound stated
therein.

\section{Concluding remarks}\label{sec:remarks}

\subsection{}
From among all 64 326 024 nontrivial snarks of order 36 or
less, generated by Brinkmann et al. \cite{BGHM}, only 621
graphs do not satisfy the assumption of
Theorem~\ref{thm:pentagons}. This may have one of two reasons:
either every edge lying on a pentagon is removable (618 graphs)
or the snark in question has girth at least $6$ (the remaining
three graphs: the double star snark of order 30 and the flower
snarks $J_7$ and $J_9$ of order 28 and 36, respectively). The
condition is therefore satisfied by approximately 99.999035
percent of all nontrivial snarks of order up to 36.

\subsection{}
Following the terminology introduced by Jaeger
\cite{Jae-strong}, we call a snark $G$ \emph{strong} if it has
no non-removable pair of adjacent vertices; that is to say, if
for any two adjacent vertices $u$ and $v$ of $G$ the graph
$G-\{u,v\}$ is not $3$-edge-colourable. Equivalently, $G$ is
strong if and only if each edge of $G$ is non-suppressible. We
say that a snark is \emph{weak} if it is not strong.

Observe that Theorem~\ref{thm:pentagons} only applies to weak
snarks as  it requires a snark in question to have a
non-suppressible edge lying on a pentagon. As shown above, such
a snark admits a cycle cover of length at most
$4/3\cdot|E(G)|+1$ and, in particular, it satisfies the
$7/5$-conjecture.

In \cite[Theorem~1.5]{LHLZ}, Liu et al. came up with a
different sufficient condition for weak snarks to satisfy the
$7/5$-conjecture. The condition of Liu et al. uses the concept
of a strong path connecting two odd circuits of a $2$-factor.
This concept was introduced (using a different terminology) by
M\'a\v cajov\'a and \v Skoviera in \cite{MS-Comb} as follows.

In a snark $G$, let $F$ be a $2$-factor with two odd circuits
$C_1$ and $C_2$ and all other circuits even, and let $M$ be the
$1$-factor complementary to $F$. Construct a matching $L$ in
$G$ by including every second edge of each even circuit of~$F$
arbitrarily. In this situation, every edge $e\in\delta_G(C_1)$
gives rise to a unique maximal $M$-$L$-alternating path $P_e$
in $G$ starting at a vertex of $C_1$ and containing $e$.
Clearly, the other endvertex of $P_e$ must lie in $C_1\cup
C_2$. If $P_e$ terminates in $C_2$, we say that $P_e$ is a
\emph{strong} $C_1$-$C_2$-path; more precisely, it is an
\emph{$(M,L)$-strong} $C_1$-$C_2$-path. Liu et al. \cite{LHLZ}
proved that if, for a suitable choice of matchings $M$ and $L$,
a snark $G$ contains two $(M,L)$-strong $C_1$-$C_2$-paths one
of which has length $1$, then $G$ admits a cycle cover of
length at most $4/3\cdot|E(G)|+2$.

Every snark satisfying the latter condition is weak, because
the vertices connected by the strong $C_1$-$C_2$-path of length
$1$ constitute a non-removable pair of adjacent vertices.
However, multiple choices for the matchings $M$ and $L$ make
the sufficient condition stated in Theorem~1.5 of \cite{LHLZ}
rather difficult to verify. Indeed, if $M$ and $L$ are not
suitably chosen, the path $P_e$ need not terminate in~$C_2$.

The fact that both Theorem~1.5 of \cite{LHLZ} and our
Theorem~\ref{thm:pentagons} only apply to weak snarks appears
to be a significant drawback. Indeed, among cubic graphs, every
potential counterexample to the $7/5$-conjecture must be a
snark with perfect matching at least~$5$. We have performed
extensive computations involving thousands of nontrivial snarks
arising from various constructions of snarks with perfect
matching index at least~$5$ (including the constructions
described in \cite{MS_P1, MS_P2}, and many others).
Computational  results reveal that all snarks that we have
tested are strong.

In view of these findings we propose the following conjecture.

\begin{conjecture}
{\rm With the only exception of the Petersen graph, every
nontrivial snark with perfect matching index at least $5$ is
strong, that is, all its edges are suppressible.}
\end{conjecture}

This conjecture does not extend to snarks with low cyclic
connectivity. Indeed, snarks arising from
Theorem~\ref{thm:main} have perfect matching index $5$, but
they cannot be strong since each doubly covered edge in a
hexagonal core is easily seen to be non-removable.

\subsection{}
There exist several constructions of infinite families of
graphs with covering ratio reaching the expected maximum of
$7/5$. All the members of these families have cyclic
connectivity 2 or 3 (\cite{BJJ, MS_SCC}). We believe that the
following is true.

\begin{conjecture}
{\rm Up to isomorphism, the Petersen graph is the only
cyclically $4$-edge-connected cubic graph with cycle covering
ratio $7/5$.}
\end{conjecture}

\section*{Acknowledgements}
\noindent{}The authors of this article were partially supported
by the grant No.~APVV-23-0076 of the Slovak Research and
Development Agency. The first author and the third author were
also supported by the grant VEGA~2/0056/25 of Slovak Ministry
of Education. The second author and the fourth author were
supported by the grant VEGA~1/0727/22 of Slovak Ministry of
Education.

\bigskip{}\smallskip


\bigskip

\begin{thebibliography}{99}
\frenchspacing

\bibitem{AT} N. Alon, M. Tarsi, \emph{Covering multigraphs
    by simple circuits}, SIAM J. Algebraic Discrete Meth.
    \textbf{6} (1985), 345--350.

\bibitem{BJJ} J.-C. Bermond, B. Jackson, F. Jaeger,
    \emph{Shortest covering of graphs with cycles}, J. Combin.
    Theory Ser. B \textbf{35} (1983), 297--308.

\bibitem{Brad1} R. C. Bradley, \emph{A remark on noncolorable
    cubic graphs}, J. Combin. Theory Ser. B \textbf{24} (1978),
    311--317.

\bibitem{Brad2} R. C. Bradley, \emph{On the number of colorings
    of a snark minus an edge}, J. Graph Theory \textbf{51}
    (2006), 251--259.

\bibitem{Brad3} R. C. Bradley, \emph{Snarks from a K\'aszonyi
    perspective: A survey}, Discrete Appl. Math. \textbf{189}
    (2015), 8--29.

\bibitem{BGHM} G.~Brinkmann, J.~Goedgebeur, J.~H\"agglund,
    K.~Markstr\"om, \emph{Generation and properties of snarks},
    J.~Combin.~Theory Ser.~B \textbf{103} (2013), 468--488.

\bibitem{CCW} P.~J.~Cameron, A.~G.~Chetwynd, J.~J.~Watkins,
    \emph{Decomposition of snarks}, J.~Graph Theory
    \textbf{11} (1987), 13--19.

\bibitem{EM} L. Esperet, G. Mazzuoccolo, \emph{On cubic
    bridgeless graphs whose edge-set cannot be covered by four
    perfect matchings}, J. Graph Theory \textbf{77} (2014),
    144--157.

\bibitem{FMS-survey} M. A. Fiol, G. Mazzuoccolo, E. Steffen,
	\emph{Measures of edge-uncolorability of cubic graphs},
	Electron.~J.~Combin.~\textbf{25} (2018), $\#$P4.54.
	
\bibitem{F} H. Fleischner, \emph{Eulerian Graphs and Related
    Topics},Part~1, Vol.~1, Ann. Discrete Math 45, North
    Holland, Amsterdam, 1990.

\bibitem{Jae-strong} F. Jaeger,  \emph{A survey of the cycle
    double cover conjecture}. North-Holland Mathematics
    Studies \textbf{115} (1985), 1--12.

\bibitem{JSw} F. Jaeger, T. Swart,  \emph{Problem session},
    Ann. Discrete Math. \textbf{9} (1980), 304--305. 	

\bibitem{JT} U. Jamshy, M. Tarsi, \emph{Short cycle covers and
    the cycle double cover conjecture}, J. Combin. Theory Ser.
    B \textbf{56} (1992), 197--204.

\bibitem{KMNS-Berge} J. Karab\'a\v s, E. M\'a\v cajov\'a, R.
    Nedela, M. \v Skoviera, \emph{Berge's conjecture for cubic
    graphs with small colouring defect}, J. Graph Theory,
    accepted; arXiv:2210.13234 [math.CO].

\bibitem{KMNS-defred} J. Karab\'a\v s, E. M\'a\v cajov\'a, R.
    Nedela, M. \v Skoviera, \emph{Cubic graphs with colouring
    defect $3$}, Electron. J. Combin. \textbf{31} (2024),
    $\#$P1.51.

\bibitem{Ka1-ortho} L. K\'aszonyi, \emph{A construction of
    cubic graphs, containing orthogonal edges}, Ann. Univ. Sci.
    Budapest E\"otv\"os Sect. Math. \textbf{15} (1972), 81--87.

\bibitem{Ka2-nonplan} L. K\'aszonyi, \emph{On the nonplanarity
    of some cubic graphs}, Ann. Univ. Sci. Budapest E\"otv\"os
    Sect. Math. \textbf{15} (1972), 123--131.

\bibitem{Ka3-struct} L. K\'aszonyi, \emph{On the structure of
    coloring graphs}, Ann. Univ. Sci. Budapest E\"otv\"os Sect.
    Math. \textbf{16} (1973), 25--36.

\bibitem{LHLZ} S. Liu, R.-X. Hao, R. Luo, C.-Q. Zhang,
    \emph{Five-cycle double cover and shortest cycle cover}, J.
    Graph Theory \textbf{108} (2025), 39--49.

\bibitem{MS-Comb} E. M\'a\v cajov\'a, M. \v Skoviera,
    \emph{Sparsely intersecting perfect matchings in cubic
    graphs}, Combinatorica~\textbf{34} (2014), 61--94.

\bibitem{MS_P1} E. M\'a\v cajov\'a, M. \v Skoviera, \emph{Cubic
    graphs that cannot be covered with four perfect matchings},
    J.~Combin. Theory Ser.~B \textbf{150} (2021), 144--176.

\bibitem{MS_P2} E. M\'a\v cajov\'a, M. \v Skoviera,
    \emph{Perfect matching index versus circular flow number of
    a cubic graph}, SIAM J. Discrete Math. \textbf{35} (2021),
    1287--1297.

\bibitem{MS_SCC} E. M\'a\v cajov\'a, M. \v Skoviera,
    \emph{Cubic graphs with no short cycle covers}, SIAM J.
    Discrete Math. \textbf{35} (2021),  2223--2233.

\bibitem{NS-decred} R.~Nedela, M.~\v Skoviera,
    \emph{Decompositions and reductions of snarks}, J.~Graph
    Theory~\textbf{22} (1996), 253--279.

\bibitem{Raspaud-diss} A. Raspaud, \emph{Flots et couvertures
    par des cycles dans les graphes et les matro\"\i des},
    Th\`ese de $3^{\text{\`eme}}$ cycle, Universit\'e de
    Grenoble, 1985.

\bibitem{S1} E. Steffen, \emph{Classifications and
    characterizations of snarks}, Discrete~Math.~\textbf{188}
    (1998), 183--203.

\bibitem{S2} E. Steffen, \emph{$1$-Factor and cycle covers of
cubic graphs}, J.~Graph~Theory~\textbf{78} (2015), 195--206.
\end{thebibliography}
\end{document}